\documentclass{article}
\usepackage[a4paper,outer=3.5cm,inner=3.5cm,total={14cm,21.4cm}]{geometry}

\usepackage[english]{babel}
\usepackage[utf8]{inputenc}
\usepackage{amsmath}
\usepackage{enumerate}
\usepackage{graphicx}
\usepackage{amssymb,amsxtra}
\usepackage{courier}
\usepackage{listings}
\usepackage{mathtools}
\RequirePackage{amsthm}
\RequirePackage{tikz-cd}
\RequirePackage{mathrsfs}
\RequirePackage[colorinlistoftodos]{todonotes}

\usepackage{amsthm}
\usepackage{amsfonts}
\usepackage{amssymb}
\usepackage{amscd}
\usepackage{epsf}
\usepackage{verbatim}
\usepackage{mathrsfs}
\usepackage{latexsym}
\usepackage{stackengine}
\usepackage{adjustbox}
\usepackage{thmtools}
\usepackage{thm-restate}
\stackMath
\usepackage{array}
\usepackage{lscape}
\usepackage[colorlinks=true]{hyperref}
\hypersetup{colorlinks, citecolor=blue, filecolor=black, linkcolor=black, urlcolor=black}
\usepackage{epstopdf}
\usepackage{tikz}
\usetikzlibrary{calc}
\usetikzlibrary{matrix,arrows,decorations.pathmorphing}
\usepackage{tikz-cd}
\usepackage{color}
\usepackage{geometry}
\usepackage{multirow}
\usepackage[shortlabels]{enumitem}

\usepackage{mdframed}
\usepackage{stmaryrd} % provides \xleftarrow

\usepackage{calligra,mathrsfs} % for typeseeting inner Hom, from https://tex.stackexchange.com/questions/141434/how-to-type-sheaf-hom/156701

\usepackage[capitalize]{cleveref}

\newtheorem{thm}{Theorem}[section]
\newtheorem{Theorem}[thm]{Theorem}
\newtheorem{Lemma}[thm]{Lemma}

\newtheorem{Proposition}[thm]{Proposition}

\newtheorem{Corollary}[thm]{Corollary}

\theoremstyle{definition}

\newtheorem{Definition}[thm]{Definition}
\newtheorem{question}[thm]{Question}

\newtheorem{Example}[thm]{Example}

\newtheorem{Remark}[thm]{Remark}

\newlist{propenum}{enumerate}{1}
\setlist[propenum]{label=(\roman*), ref=\theProposition.(\roman*)}
\crefalias{propenumi}{Proposition}

\newlist{defenum}{enumerate}{1}
\setlist[defenum]{label=(\alph*), ref=\theDefinition.(\alph*)}
\crefalias{defenumi}{Definition}

\newlist{lemenum}{enumerate}{1}
\setlist[lemenum]{label=(\roman*), ref=\theLemma.(\roman*)}
\crefalias{lemenumi}{Lemma}

\newlist{remarksenum}{enumerate}{1}
\setlist[remarksenum]{label=(\roman*), ref=\theRemarks.(\roman*)}
\crefalias{remarksenumi}{Remark}

\newcommand{\smallmat}[4]{\left(\begin{smallmatrix} #1 & #2 \\ #3 & #4\end{smallmatrix}\right)}

\newcommand*\isomarrow{%
	\xrightarrow{\raisebox{-0.35em}{\smash{\ensuremath{\sim}}}}
}  
\newcommand{\xisomarrow}[1]{%
	\xrightarrow[\raisebox{0.3em}{\smash{\ensuremath{\sim}}}]{#1}
}

\newcommand{\wh}{\widehat}
\newcommand{\LSD}{\mathrm{LSD}}
\newcommand{\pf}{{\tt}}

\newcommand{\xtofrom}[1]{\mathbin{\xleftrightarrow{#1}}}
\newcommand{\isotofrom}{\xtofrom\sim}

\tikzset{
	labelrotate/.style={anchor=south, rotate=90, inner sep=.5mm}}

\newcommand{\N}{\mathbb{N}}
\newcommand{\Z}{\mathbb{Z}}
\newcommand{\Q}{\mathbb{Q}}

\newcommand{\C}{\mathbb{C}}

\newcommand{\B}{\mathbb{B}}

\renewcommand{\O}{\mathcal{O}}

\newcommand{\G}{\mathbb{G}}

\newcommand{\Hom}{\operatorname{Hom}}

\newcommand{\Ext}{\operatorname{Ext}}
\newcommand{\End}{\operatorname{End}}
\newcommand{\IEnd}{\operatorname{\underline{End}}}
\newcommand{\Aut}{\operatorname{Aut}}
\newcommand{\GL}{\operatorname{GL}}

\newcommand{\Perf}{\operatorname{Perf}}

\newcommand{\Spa}{\operatorname{Spa}}

\newcommand{\id}{{\operatorname{id}}}

\newcommand{\cts}{{\operatorname{cts}}}

\newcommand{\an}{{\mathrm{an}}}

\newcommand{\et}{{\operatorname{\acute{e}t}}}
\newcommand{\proet}{{\operatorname{pro\acute{e}t}}}
\newcommand{\profet}{{\operatorname{prof\acute{e}t}}}

\newcommand{\HT}{\operatorname{HT}}

\newcommand{\comp}{\mathbin{\circ}}

\newcommand{\Pic}{\operatorname{Pic}}
\newcommand{\Higgs}{\operatorname{Higgs}}

\newcommand{\wt}{\widetilde}

\newcommand{\rk}{\mathrm{rk}}
\renewcommand{\wh}{\widehat}

\renewcommand{\lim}{\varprojlim}

\newcommand{\cH}{{\ifmmode \check{H}\else{\v{C}ech}\fi}}

\renewcommand{\tt}{\mathrm{tt}}

\newcommand{\aeq}{\stackrel{a}{=}}

\renewcommand{\O}{\mathcal O}

\newcommand{\tensor}{\mathbin{\otimes}}
\newcommand{\wtOm}{\widetilde{\Omega}}
\renewcommand{\diamond}{\diamondsuit}

\newcommand{\xto}[1]{\mathbin{\xrightarrow{#1}}} % TODO: The spacing seems incorrect
\newcommand{\injto}{\mathrel{\hookrightarrow}}
\newcommand{\surjto}{\mathrel{\twoheadrightarrow}}

\DeclareMathOperator{\HOM}{\mathscr{H}\text{\kern -3pt {\calligra\large om}}\,}
\DeclareMathOperator{\EXT}{\mathscr{E}\text{\kern -3pt {xt}}\,}
\makeatletter
\newcommand{\doublewidetilde}[1]{{%
		\mathpalette\double@widetilde{#1}%
}}
\newcommand{\double@widetilde}[2]{%
	\sbox\z@{$\m@th#1\widetilde{#2}$}%
	\ht\z@=.85\ht\z@
	\widetilde{\box\z@}%
}

\let\old@widetilde\doublewidetilde

\makeatother

\begin{document}
	\title{The $p$-adic Corlette--Simpson correspondence for abeloids}
	\author{Ben Heuer, Lucas Mann, Annette Werner}
	\date{}
	\maketitle
	\begin{abstract}
		For an abeloid variety $A$ over a complete algebraically closed field extension $K$ of $\Q_p$, we construct a $p$-adic Corlette--Simpson correspondence, namely an equivalence between finite-dimensional continuous $K$-linear representations of the Tate module and a certain subcategory of the Higgs bundles on $A$.  To do so, our central object of study is the category of vector bundles for the $v$-topology on the diamond associated to $A$. We prove that any pro-finite-\'etale $v$-vector bundle can be built from pro-finite-\'etale $v$-line bundles and unipotent $v$-bundles. To describe the latter,  we extend the theory of universal vector extensions to the $v$-topology and use this to generalise a result of Brion by relating unipotent $v$-bundles on abeloids to representations of vector groups.
	\end{abstract}
	\centerline{\small \textbf{2020 MSC:} 14K15, 14G45, 14G22
	}
	\setcounter{tocdepth}{2}

	\tableofcontents

	%------------------------------------------------------------------------------
	\section{Introduction}
	Let $K$ be a complete algebraically closed field extension of $\Q_p$. The goal of this article is to prove the following version of a $p$-adic Corlette--Simpson theorem for abeloid varieties, i.e.\ connected smooth proper commutative rigid group varieties over $K$ (see \cref{rslt:Simpson-correspondence}):
	
	\begin{Theorem} \label{rslt:intro-Simpson-correspondence-vers2}
		Let $A$ be an abeloid variety over $K$. Then there is an equivalence of categories
		\begin{align*}
			\Bigg\{\begin{array}{@{}c@{}l}\text{finite-dimensional continuous  }\\
				\text{$K$-linear representations of $\pi_1(A,0)$}\end{array}\Big\}\cong  \Bigg\{\begin{array}{@{}c@{}l}\text{ pro-finite-\'etale}\\
				\text{Higgs bundles on $A$}\end{array}\Bigg\}.
		\end{align*}
		This is canonical after choices of an exponential on $K$ and of a splitting of the Hodge--Tate sequence of $A$.
	\end{Theorem}
	Here pro-finite-\'etale Higgs bundles are defined as Higgs bundles such that the underlying vector bundle becomes trivial on a pro-finite-\'etale covering of $A$.

	To motivate the  theorem, we recall that the complex Corlette--Simpson correspondence for a compact Kähler manifold  $X$ provides a category equivalence between  finite-dimensional complex representations of the fundamental group of $X$ and semi-stable Higgs bundles on $X$ with vanishing Chern classes \cite{SimpsonCorrespondence}.

	In the past 15 years, there has been extensive work on establishing analogs of the complex Corlette--Simpson correspondence in the $p$-adic setting, i.e.\ for proper smooth varieties over $\C_p$ (or more generally for proper smooth rigid spaces over any complete algebraically closed extension $K$ of $\Q_p$): Faltings \cite{Faltings_SimpsonI} and later Abbes--Gros--Tsuji \cite{AGT} have established a correspondence between small Higgs bundles and small so-called generalised representations for smooth proper non-archimedean varieties with a toroidal model. Faltings also proves a result for curves without the smallness assumptions on both sides, but it seems that it is currently not known whether these assumptions can also be eliminated  in higher dimensions.

	Despite all advances, some of which we review below, a major question in the field remains open, namely the problem to determine which Higgs bundles correspond to genuine representations of the \'etale fundamental group: For algebraic varieties over $\C_p$, one might expect that numerical flatness plays a role. For rigid analytic varieties over general base fields,
	it seems that there is currently not even a precise conjecture for what the correct subcategory of Higgs bundles for such a $p$-adic Corlette--Simpson correspondence should be.
	
	\paragraph{The case of abeloid varieties.}
	The purpose of this paper is to solve the aforementioned problem in the case of abeloid varieties. This gives a first idea what kind of statement to expect in general.

	We note that \cref{rslt:intro-Simpson-correspondence-vers2} is closely analogous to the complex case: On complex abelian varieties, semi-stable bundles with vanishing Chern classes are precisely the homogeneous ones
	\cite[Theorem~2]{MehtaNori}. Since on an abelian variety $A$  over  $K = \C_p$ the pro-finite-étale vector bundles coincide with the homogeneous ones (see \cref{rem:homogeneous}), our result \cref{rslt:intro-Simpson-correspondence-vers2} in this case is indeed analogous to the complex Corlette--Simpson correspondence. However, interestingly, this immediate analogy breaks down if the base field is larger than $\C_p$, in which case being pro-finite-\'etale turns out to be stronger than being homogeneous.
	
	Our proof of \cref{rslt:intro-Simpson-correspondence-vers2} relies heavily on the machinery of the pro-\'etale site of Scholze's $p$-adic Hodge theory \cite{Scholze_p-adicHodgeForRigid}. More generally, it will be convenient to work in the framework of diamonds \cite{etale-cohomology-of-diamonds} and the $v$-site that comes with it. This allows us to reinterpret the representations on the left-hand side of the main result as vector bundles on the $v$-site, which are very closely related to Faltings' generalised representations. This  reduces \cref{rslt:intro-Simpson-correspondence-vers2} to a study of the category of $v$-vector bundles on $A$.
	
	\paragraph{$v$-Vector Bundles on Abeloid Varieties.} The category of (locally spatial) diamonds over $K$ is a large category of ``spaces'' which contains all smooth rigid $K$-varieties and all perfectoid spaces over $K$ as full subcategories and hence provides a convenient framework for our purposes. Diamonds come equipped with a topology (i.e.\ an analytic site) and an étale site, which match the respective sites of smooth rigid $K$-varieties, thereby allowing us to identify any smooth rigid $K$-variety with its associated diamond without losing any information. But they also provide us with the much finer $v$-site: For any locally spatial diamond $X$, the $v$-site $X_v$ is the site of all locally spatial diamonds over $X$, where coverings are jointly surjective maps satisfying some quasi-compactness hypothesis. The site $X_v$ comes equipped with a structure sheaf $\O_X$, which enables us to consider vector bundles on $X_v$:
	
	\begin{Definition}
		Let $X$ be a locally spatial diamond over $K$ (e.g.\ a smooth rigid $K$-variety).
		\begin{defenum}
			\item A \emph{$v$-vector bundle} on $X$ is a sheaf of $\O_X$-modules on $X_v$ which is $v$-locally isomorphic to $\O_X^n$ for some $n \in \Z_{\ge0}$.
			
			\item A $v$-vector bundle $E$ on $X$ is called \emph{pro-finite-étale} if it becomes trivial on a pro-finite-étale cover of $X$, i.e.\ a cofiltered inverse limit of finite \'etale covers.
		\end{defenum}
	\end{Definition}
	
	Via pullback along $X_v\to X_{\an}$, the category of $v$-vector bundles contains the category of classical vector bundles on $X_{\an}$ as a full subcategory. To avoid confusion, we will call objects of the latter category \emph{analytic} vector bundles.
	
	Let now $X$ be a connected proper smooth rigid $K$-variety and fix any point $x \in X(K)$. Then using the language of $v$-vector bundles, we get a new interpretation of the category of $K$-linear representations of $\pi_1(X,x)$, so in particular of the left-hand-side of \cref{rslt:intro-Simpson-correspondence-vers2}. Namely, the language of diamonds allows us to construct the universal pro-finite-étale cover
	\begin{align*}
		\wt X := \varprojlim_{X' \to X} X',
	\end{align*}
	where the limit is taken over all pairs $(X', x')$ of connected finite étale covers $X' \surjto X$ together with a point $x' \in X'(K)$ over $x$. Then $\wt X \to X$ is a $\pi_1(X, x)$-torsor and it follows (see \cref{rslt:profet-vb-vs-Reps}) that there is a natural equivalence of categories
	\begin{align*}
		\Bigg\{\begin{array}{@{}c@{}l}\text{pro-finite-\'etale  }\\
			\text{ $v$-vector bundles on $X$}\end{array}\Big\}\cong  \Bigg\{\begin{array}{@{}c@{}l}\text{ fin.-dim.\ continuous}\\
			\text{$K$-linear $\pi_1(X,x)$-rep.}\end{array}\Bigg\}.
	\end{align*}
	
	In this article, we study the category of $v$-vector bundles in the case that $X$ is an abeloid variety $A$. Abeloids are a rigid generalisation of abelian varieties that may be regarded as the $p$-adic analogs of complex tori. In close analogy with the description of homogeneous vector bundles on complex tori due to Matsushima and Morimoto \cite{Morimoto}\cite{Matsushima-Morimoto}, as well as on abelian varieties due to Miyanishi and Mukai \cite{Miyanishi}\cite[Theorem 4.17]{Mukai_semi-homog-VB_on_AV}, we prove the following structure result for pro-finite-\'etale $v$-vector bundles on abeloids
	(see \cref{rslt:vb-on-A-are-unipotent-times-lb}):
	
	\begin{Theorem} \label{rslt:intro-decomposition-of-profet-vb-vers2}
		Let $A$ be an abeloid variety over $K$. Then a $v$-vector bundle $E$ on $A$ is pro-finite-\'etale if and only if it decomposes as a direct sum
		\begin{align*}
			E = \bigoplus_{i=1}^n (U_i \tensor L_i),
		\end{align*}
		where each $U_i$ is a unipotent $v$-vector bundle (i.e.\ a successive extension of trivial bundles) and each $L_i$ is a pro-finite-\'etale line bundle on $A$.
		
		If $E$ is analytic, all $U_i$ and $L_i$ occuring in this decomposition are also analytic.
	\end{Theorem}
	
	\paragraph{Proof Strategy.} We sketch the proof of \cref{rslt:intro-Simpson-correspondence-vers2}. Using \cref{rslt:intro-decomposition-of-profet-vb-vers2} we can reduce the proof to the following two cases:
	
	\begin{enumerate}
		\item Line bundles. This was done in previous work of the first author \cite{heuer-v_lb_rigid} (see \S\ref{sec:prelim-simpson-for-line-bundles}). 
		
		\item Unipotent objects on both sides, i.e.\ successive extensions of the trivial object. This part does not rely on the choice of an exponential.
	\end{enumerate}
	
	The proof of step 2 relies on the following analog of Brion's correspondence of unipotent vector bundles and representations of the vector group $H^1(A, \O_A)^* \tensor \mathbb G_a$  (where $-^{\ast}$ denotes the $K$-vector space dual) in the algebraic setting of abelian varieties \cite[\S3.3]{Brion_homogVB_AV} (see \cref{rslt:characterize-tau-unipotent-VB-via-algebraic-reps}).
	
	\begin{Theorem} \label{rslt:intro-Brion-vers2}
		Let $A$ be an abeloid variety over $K$ and let $\tau \in \{\an, v\}$. Then there is a natural equivalence of categories
		\begin{align*}
			\{ \text{alg.\ representations of $H^1_\tau(A,\O_A)^* \tensor_K \mathbb G_a$} \} \ \cong \ \{ \text{unipotent $\tau$-vector bundles on $A$} \}.
		\end{align*}
	\end{Theorem}
	
	The proof of \cref{rslt:intro-Brion-vers2} relies on a theory of vector extensions (both for the analytic and the $v$-site) analogous to the classical theory. The appearance of $H^1_\tau(A, \O_A)^* \tensor \mathbb G_a$ comes from the fact that this vector group features in the universal $\tau$-vector extension of $A$.
	This is shown in \cref{rslt:universal-vector-extension-exists}, which generalises the classical result on the universal vector extensions from abelian to abeloid varieties and from the analytic to the $v$-topology.
	
	In order to link unipotent $v$-vector bundles to unipotent Higgs bundles, recall that on every smooth proper rigid space $X$ over $K$ there is a canonical short exact sequence
	\begin{align}
		0 \to H^1_\an(X, \O_X) \to H^1_v(X, \O_X) \xto{\HT} H^0(X, \Omega_X^1(-1)) \to 0 \label{eq:intro-HT}
	\end{align}
	of finite dimensional $K$-vector spaces, the Hodge-Tate sequence (see \cref{sec:prelim-higgs-bundles}). Now apply this to $X = A$ and fix a splitting of this sequence. This induces an isomorphism
	\begin{align*}
		H^1_v(A, \O_A)^* \cong H^1_\an(A, \O_A)^* \oplus H^0(A, \Omega_A^1(-1))^*.
	\end{align*}
	By \cref{rslt:intro-Brion-vers2}, representations of the left-hand-side correspond to unipotent $v$-vector bundles, while representations of the right-hand-side correspond to unipotent analytic vector bundles together with a commuting action of $H^0(A, \Omega_A^1(-1))^*$. It is not hard to verify that the latter action corresponds precisely to a Higgs field (see \cref{rslt:Simpson-correspondence-for-unipotent}). This concludes the proof of \cref{rslt:intro-Simpson-correspondence-vers2}.

	\paragraph{Naturality of the Correspondence.} 
	In contrast to the classical result over the complex numbers, the $p$-adic Corlette--Simpson correspondence is in general expected to depend on a splitting of the Hodge--Tate sequence and on a choice of an exponential on $K$.
	
	To understand why the first choice appears, one should view the $v$-cohomology $H^1_v(X, \O_X)$ appearing in \cref{eq:intro-HT} as an analog of the singular cohomology $H^1(Y, \C)$ on a complex manifold $Y$. In fact, by Scholze's Primitive Comparison Theorem we have
	\begin{align*}
		H^1_v(X, \O_X) = H^1_{\et}(X, \Z_p) \tensor_{\Z_p} K.
	\end{align*}
	In the complex case, the existence of a Kähler form guarantees a canonical splitting of the analog of \cref{eq:intro-HT} (the Hodge decomposition), which plays an important role in the Corlette--Simpson correspondence. In contrast, there is in general no canonical splitting in the non-archimedean world, so we content ourselves with artificially choosing any splitting of the Hodge-Tate map $\HT$. A better way to make this choice is arguably to choose a lift of $A$ to $B_{\mathrm{dR}}^+/\xi^2$, which by \cite[Proposition 7.2.5]{guo2019hodgetate} induces a splitting of this sequence. In particular, by considering the category of abeloids equipped with such a lift, the correspondence in \cref{rslt:intro-Simpson-correspondence-vers2} becomes natural in $A$. We also note that if $A$ is defined over a finite extension of $\Q_p$, there is a canonical choice of such a lift, and thus of a splitting of $\HT$.
	
	The choice of an exponential results from the step from $\G_a$ to $\G_m$, and essentially only affects the correspondence for line bundles since it does not enter in the unipotent case. Indeed, given any two choices of an exponential, the resulting equivalences on indecomposable objects only differ by twists with analytic torsion line bundles.
	
	\paragraph{Related Work.}
	Towards establishing a $p$-adic Corlette--Simpson correspondence for representations of the \'etale fundamental group, 
	Deninger and the third author have investigated the case of Higgs bundles with vanishing Higgs field, in which case they show that one can attach representations to vector bundles with numerically flat reduction \cite{DeningerWerner_vb_p-adic_curves,DeningerWerner_vb_p-adic_varieties}. W\"urthen \cite{wuerthen_vb_on_rigid_var} has extended this functor to the rigid analytic case and has shown that for analytic vector bundles, the notion of numerically flat reduction is closely related to being pro-finite-\'etale, i.e.\ the condition that also appears in \cref{rslt:intro-Simpson-correspondence-vers2}. For abeloids, it is easy to see that the functor thus defined agrees with ours restricted to vanishing Higgs fields.

	Liu--Zhu \cite{LiuZhu_RiemannHilbert} investigate a Riemann--Hilbert functor on a smooth rigid analytic variety $X$ over a finite extension of $\mathbb{Q}_p$, which yields part of a $p$-adic Corlette--Simpson correspondence, namely a tensor functor from the category of \'etale $\mathbb{Q}_p$-local systems on $X$ to the category of nilpotent Higgs bundles on $X \tensor \mathbb{C}_p$. Their approach also uses the pro-\'etale site in an essential way, but it starts with an arithmetic datum on the representation side and does not consider the question when a local system can be attached to a Higgs bundle.

	The category $v$-bundles has previously been studied by the second and third author \cite{MannWerner_LocSys_p-adVB}: For example, they show that pro-finite-\'etaleness and related properties of vector bundles can be checked on proper covers.
	
	One case in which a $p$-adic Corlette--Simpson correspondence for representations is known is the case of rank one \cite{heuer-v_lb_rigid}: In this case, characters of the \'etale fundamental group correspond to pro-finite-\'etale Higgs line bundles, and one can also explicitly describe pro-finite-\'etale line bundles as topological torsion points in the Picard variety \cite{heuer-diamantine-Picard}.
	
	However, a ``full'' $p$-adic Corlette--Simpson correspondence from a specific subcategory of Higgs bundles on a proper smooth variety $X$ to the category of finite-dimensional continuous $K$-linear representations of the étale fundamental group of $X$ has so far not been established yet, not even in non-trivial special cases of $X$.
	In particular, the question which vector bundles correspond to genuine $p$-adic representations in Faltings' equivalence remains open.
	
	\paragraph{Outlook.} We believe that our results are useful in two ways:     Firstly, we believe that studying abeloid and abelian varieties from a diamantine point of view gives interesting new insights in their theory, for example the universal $v$-vector extension in this case.
	
	Secondly, since currently there is not even a conjecture for what the $p$-adic Corlette--Simpson correspondence for representations of the \'etale fundamental group should look like, it seems very important to first understand interesting special cases. 
	Similarly to how $p$-adic Hodge theory was first fully understood on abelian varieties, the goal of this article is therefore to fully understand the $p$-adic side of the Corlette--Simpson correspondence (i.e.\ ``non-abelian $p$-adic Hodge theory'') for abelian varieties, and more generally abeloids.  Of course, the fact that the \'etale fundamental group is abelian in this case makes the correspondence much more accessible. However, we note that the investigation of the category of $v$-bundles in this special case already has implications for much more general situations.
	
	For example, for algebraic $X$, or by results of Hansen--Li \cite[Proposition 4.3]{Hansen_Li} also for rigid analytic $X$ with projective reduction, there is an ``Albanese abeloid'' $X\to A$ with $H^0(A,\Omega^1)=H^0(X,\Omega^1)$. It follows that our abeloid description captures the entire ``abelian'' part of the $p$-adic Corlette--Simpson correspondence for $X$, i.e.\ those representations that factor through the torsionfree quotient of the abelianization of the fundamental group of $X$ (which is the adelic Tate module $\pi_1(A,0)$ of the Albanese abeloid). In particular, this explains in this generality which representations should be associated to Higgs fields on the trivial bundle on $X$. 
	
	\paragraph{Acknowledgements.} We would like to thank Peter Scholze for many helpful discussions, and the referee for their careful reading and very useful comments.
	
	The first and second author were each supported by the Deutsche Forschungsgemeinschaft (DFG, German Research Foundation) via the Leibniz-Preis of Peter Scholze. 
	The first author was funded by DFG under Germany's Excellence Strategy-- EXC-2047/1 -- 390685813. The third author acknowledges support by the Deutsche Forschungsgemeinschaft, through TRR 326 \textit{Geometry and Arithmetic of Uniformized Structures}, project number 444845124.

	%------------------------------------------------------------------------------
	
	%------------------------------------------------------------------------------
	\section{Preliminaries on $v$-vector bundles}
	
	Throughout this article, we fix a complete algebraically closed non-archimedean field extension $K$ of $\Q_p$. We will work with smooth proper rigid spaces $X$ over $K$; later $X$ will always be an abeloid variety, but for now we can work in the general case.
	
	\subsection{Recollections on diamonds} \label{sec:prelim-diamonds}
	
	It will be important for our considerations to regard $X$ as a diamond in the sense of Scholze. In the following we recall some technical background for this.
	
	Let $\Perf_{K}$ the category of perfectoid spaces over $K$. This carries various natural topologies; we will be most interested in the $v$-topology in the sense of \cite[Definition~8.1]{etale-cohomology-of-diamonds}. We refer to \cite[\S11]{etale-cohomology-of-diamonds} for the definition of diamonds; we will exclusively work with ``locally spatial diamonds''. By \cite[Theorem~12.18]{etale-cohomology-of-diamonds} these can roughly be described as $v$-sheaves $Y$ on $\Perf_K$ that admit a quasi-pro-\'etale surjection by a perfectoid space $X$. We denote by $\LSD_K$ the full subcategory of $v$-sheaves on $\Perf_K$ consisting of all locally spatial diamonds. Every locally spatial diamond $Y$ comes equipped with the following sites:
	\begin{itemize}
		\item The $v$-site $Y_v$ consisting of all locally spatial diamonds over $Y$ with the $v$-topology. This definition differs slightly from \cite[Definition 14.1]{etale-cohomology-of-diamonds}, where $Y_v$ consists of all small $v$-sheaves on $\Perf_K$ over $Y$ -- however, both definitions produce the same topos, so the distinction does not matter in practice. In a similar vein we could also replace $Y_v$ by the $v$-site of perfectoid spaces over $Y$.
		
		\item The analytic site $Y_\an$, which is the site associated to the topological space $|Y|$ (see \cite[Definition 11.14]{etale-cohomology-of-diamonds}). 
		By definition, the condition that $Y$ be locally spatial means that $|Y|$ is locally spectral.
		
		\item The étale site $Y_\et$ and the pro-étale site $Y_\proet$ (see \cite[\S14]{etale-cohomology-of-diamonds}).
	\end{itemize}
	Given an analytic adic space $Z$ over $K$, one associates a presheaf $Z^\diamondsuit$ on $\Perf_{K}$ by sending any perfectoid space $X$ over $K$ to the set of morphisms of adic spaces $X\to Z$ over $K$. By the tilting equivalence, this coincides with the definition in \cite[Definition 15.5]{etale-cohomology-of-diamonds}.

	\begin{Theorem}[{\cite[Theorem 15.6]{etale-cohomology-of-diamonds}}, {\cite[Theorem 8.2.3]{relative-p-adic-hodge-2}}]\label{rslt:dmd-fully-faithful}
		$Z^\diamondsuit$ is a diamond, and in particular a $v$-sheaf. This ``diamondification'' defines a fully faithful functor
		\[\{ \text{semi-normal rigid spaces over } K \}\hookrightarrow \mathrm{LSD}_K,\quad X\mapsto X^\diamond.\]
		Moreover, for every semi-normal rigid space $X$ over $K$ we have natural equivalences of sites $X_\an = X^\diamond_\an$ and $X_\et = X^\diamond_\et$.
	\end{Theorem}
	
	For any analytic adic space over $K$, diamondification defines a natural morphism of sites
	\[\nu\colon X^\diamond_v \to X_\an.\]
	There are structure sheaves $\O_{v}$ and $\O^+_{v}$ on $\Perf_K$ which are sheaves for the $v$-topology \cite[Theorem 8.7]{etale-cohomology-of-diamonds}. In particular, they naturally extend to $v$-sheaves on $\LSD_K$. If $X$ is a semi-normal rigid space, we will usually abuse notation and write $X_v$ instead of $X^\diamond_v$.  On $X$ the restriction $\nu_{\ast}\O_v$ of $\O_v$ to $X_{\an}$ coincides with the analytic structure sheaf $\O_{X_{\an}}$ via the natural map \cite[Theorem 8.2.3]{relative-p-adic-hodge-2}.
	See \cite[\S2]{MannWerner_LocSys_p-adVB} for more details. We will therefore often omit the subscript and simply write $\O$ and $\O^+$.
	
	We will write $H^\ast_{\an}(X,-)$ and $H^\ast_{v}(X,-)$ for the cohomology with respect to the (small) analytic site $X_{\an}$ and the site $X^\diamond_v$, respectively.
	
	\subsection{The diamantine universal cover}
	Next, we recall the pro-finite-\'etale universal cover of a rigid space:
	\begin{Definition}[{cf. \cite[Definition 4.6]{heuer-v_lb_rigid}}]
		Let $X$ be a connected smooth rigid space over $K$ and fix any base-point $x \in X(K)$. We define the \emph{pro-finite-étale universal cover} $\wt X$ of $X$ as
		\begin{align*}
			\wt X := \varprojlim_{X'\to X} X',
		\end{align*}
		where the index category on the right is given by the pointed maps $(X', x')\to (X,x)$ from connected finite-étale covers $X' \to X$ together with a point $x' \in X'(K)$ over $x$. The limit is taken in the category of locally spatial diamonds and in particular $\wt X$ is a locally spatial diamond (cf. \cite[Lemma 11.22]{etale-cohomology-of-diamonds}). The points $x'$ give rise to a lift $\wt x\in \wt X(K)$ of $x$. While  $\wt X\to X$ has a large automorphism group, the additional datum of this point $\wt x$ makes the pointed space $(\wt X,\wt x)$ unique up to \textit{unique} isomorphism, and functorial in $X$. 
	\end{Definition}

	It is easy to see that $\wt X$ is the universal pro-finite-\'etale cover of $X$ in the following sense:
	\begin{Lemma}[{\cite[Lemma~4.8]{heuer-v_lb_rigid}}]\label{up-1-of-univ-cover}
		Let $Y\to X$ be any pro-finite-\'etale cover with a lift $y\in Y(K)$ of $x$. Then there is a unique morphism $\wt X\to Y$ over $X$ sending $\wt x$ to $y$.
	\end{Lemma}
	
	If $X$ is proper, then $\wt X$ is also the ``universal cover'' of $X$ in a $v$-cohomological sense:
	
	\begin{Proposition} \label{rslt:properties-of-wt-X}
		Let $X$ be a connected smooth proper rigid space over $K$.
		\begin{propenum}
			\item $\wt X \to X$ is a pro-\'etale torsor under the étale fundamental group $\pi_1(X, x)$.
			\item Let $F$ be one of the following $v$-sheaves: $\Z_p$, $\widehat \Z$, $\O^{+a}$, $\O$ or any abelian torsion group $G$ considered as a constant sheaf. Then
			\begin{align*}
				H^0(\wt X, F) = F(K), \qquad H^1_v(\wt X, F) = 0.
			\end{align*}
		\end{propenum}
	\end{Proposition}
	\begin{proof}
		Part (i) follows from the fact that $\wt X \times_X \wt X \to \wt X$ is an inverse limit of finite étale covers and every finite étale cover of $\wt X$ splits and hence is isomorphic to a finite disjoint union of copies of $\wt X$. Part (ii) follows easily from \cite[Proposition 4.9]{heuer-v_lb_rigid}. As a summary, note that the cohomology of $G$ and consequently $\Z_p$ is easily computed. To get the claim about the cohomology of $\O^+$ we can then use the Primitive Comparison Theorem (\cite[Theorem 5.1]{Scholze_p-adicHodgeForRigid}). The cohomology of $\O$ follows by inverting $p$ (since $\wt X$ is qcqs, filtered colimits can be pulled out of the cohomology).
	\end{proof}
	
	\subsection{The universal cover of abeloids}
	
	If $X=A$ is an abeloid variety, then the universal cover $\wt X$ has particularly good properties: In this case we have a canonical base point $0\in A(K)$. Since any connected finite \'etale cover of $A$ is an isogeny from an abeloid $A'\to A$, we then more explicitly have
	\[ \wt A=\varprojlim_{[N]} A,\]
	and thus $\pi_1(A,0)=TA=\varprojlim_{N\in\N} A[N]$ is the adelic Tate module. In particular, the universal cover in this case gives rise to a diamantine uniformisation
	\[ A=\wt A/TA.\]
	We refer to \cite{perfectoid-covers-Arizona}\cite[\S1, \S3]{heuer-isoclasses} for a more detailed discussion of the space $\wt A$ and its properties. For us it will be important that for $\wt A$ one can improve on the properties stated in \cref{rslt:properties-of-wt-X}:
	\begin{Proposition}[{\cite[Corollary~5.8]{perfectoid-covers-Arizona}, \cite[Proposition 4.2]{heuer-Picard-good-reduction}.}] \label{rslt:vanishing-of-higher-cohom-on-wt-X-for-abeloid} 
		$\wt A$ is a perfectoid space. Moreover, we have 
		\[ H^i(\wt A,\O^+)\aeq 0 \quad \text{ for all }i\geq 1.\]
	\end{Proposition}
	\begin{Remark}
		This is not true in general, e.g. for $X = \mathbb P^n$ we simply have $\wt X = X$. It is unknown whether $\wt X$ can always be represented by an analytic adic space.
	\end{Remark}
	\subsection{Pro-finite-étale vector bundles}
	The point of view we would like to adopt on Faltings' work is to replace generalised representations by locally trivial $\O$-modules in the $v$-topology. From this perspective, the honest representations $\pi\to \GL_n(K)$ correspond precisely to those $v$-vector bundles that become trivial on the cover $\wt X \to X$, as we shall now discuss.
	
	\begin{Definition}
		Let $X$ be a smooth rigid space over $K$.
		\begin{defenum}
			\item A \emph{$v$-vector bundle} on $X$ is a locally free $\O$-module of finite rank on $X_v$. 
			
			\item An \emph{analytic vector bundle} on $X$ is a locally free $\O$-module of finite rank on $X_{\an}$. 	There is a natural functor from analytic vector bundles to $v$-vector bundles given by
			\[ M\mapsto \nu^{\ast}M:=\nu^{-1}M\otimes_{\nu^{-1}\O_{X_{\an}}}\O_v\]
			where we recall that $\nu:X_{v}\to X_{\an}$ denotes the natural morphism of sites. We say that a $v$-vector bundle is analytic if it is in the essential image of this functor.
			\item A $v$-vector bundle on $X$ is called \emph{pro-finite-\'etale} if it becomes free on a pro-finite-étale cover of $X$. If $X$ is connected, then by \cref{up-1-of-univ-cover} this is equivalent to saying that it becomes free on $\wt X$. 
			
			\item We call an analytic vector bundle pro-finite-\'etale if its associated $v$-vector bundle is pro-finite-\'etale.
		\end{defenum}
	\end{Definition}
	
	In the case of $K = \C_p$, the category of pro-finite-\'etale $v$-vector bundles was studied in \cite[\S3.1]{wuerthen_vb_on_rigid_var} and \cite{MannWerner_LocSys_p-adVB} (more precisely, pro-finite-\'etale $v$-vector bundles are equivalent to  vector bundles with properly trivializable reduction modulo all $p^n$ in the sense of \textit{loc cit}).
	
	Note that we have a diagram of sites
	
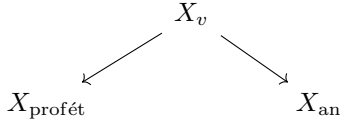
\begin{center}\begin{tikzcd}
			& X_{v} \arrow[ld] \arrow[rd] \\
			X_{\profet} && X_{\an}
	\end{tikzcd}\end{center}
	but there is no functor between the pro-finite-étale site and the analytic site. Indeed, if $X$ is connected, the only object contained in both is the identity $X\to X$.
	
	The following Lemma says that for analytic bundles, we may freely switch back and forth between $X_{\an}$ and $X_v$.
	\begin{Lemma}\label{l:an-to-v-vb-fully-faithful}
		The functor $M\mapsto \nu^{\ast}M$ is fully faithful.
	\end{Lemma}
	\begin{proof}
		Since $\nu_{\ast}\O_v=\O_{X^{\an}}$ by \cref{rslt:dmd-fully-faithful} and the following remarks, the functor $\nu_{\ast}$ defines a quasi-inverse on the essential image, as one sees locally on $X$.
	\end{proof}
	
	For any base point $x$ we denote by $\mathrm{Rep}_{K}(\pi_1(X,x))$ the category of continuous finite-dimensional $K$-linear representations of $\pi_1(X,x)$.
	The following theorem is the key to analyze this category. 
	
	\begin{Theorem}[{\cite[Theorem~5.2]{heuer-v_lb_rigid}}]\label{rslt:profet-vb-vs-Reps}
		Let $X$ be a connected smooth proper rigid space over $K$ and fix $x \in X(K)$. Then there is an exact equivalence of tensor categories
		\begin{alignat*}{3}
			\{ \text{\normalfont pro-finite-\'etale $v$-vector bundles on $X$} \} & \ \isomarrow \ && \mathrm{Rep}_{K}(\pi_1(X,x)),\\
			V& \ \, \mapsto&&V(\wt X).\\
			V_{\rho}&  \ \, \mapsfrom &&\rho:\pi_1(X,x)\to \GL(W)
		\end{alignat*}
		where $V_\rho$ is defined as the v-sheaf on $X$ that sends $Y\to X$ to
		\[V_\rho(Y)=\{ x\in W\otimes_K\O(Y\times_X\wt X)\mid g^{\ast} x=\rho^{-1}(g)x \text{ for all }g\in \pi_1(X,x)\}\]
		If $K = \C_p$, then these categories are also equivalent to the category of $\C_p$-local systems on $X_v$ which possess an integral model, as defined in \cite[Definition 3.23]{MannWerner_LocSys_p-adVB}.
	\end{Theorem}
	
	This can be regarded as an extension of \cite[\S3]{wuerthen_vb_on_rigid_var}, where a functor from analytic vector bundles to representations is constructed.
	
	\begin{proof}
		This equivalence is a formal consequence of glueing in the $v$-site together with \cref{rslt:properties-of-wt-X}. More precisely, since any pro-finite-\'etale $v$-vector bundle on $X$ is free on $\wt X$, the category of pro-finite-\'etale $v$-vector bundles is equivalent to the category of descent data of finite free $\O$-modules on $\wt X_v$ along $\wt X \to X$. From $H^0(\wt X, \O) = K$ it is clear that the category of finite free $\O$-modules on $\wt X$ is equivalent to the category of finite $K$-vector spaces. Since $\wt X \to X$ is a $\pi_1(X, x)$-torsor, it follows that giving a descent datum of a finite free $\O$-module $\mathcal M$ along $\wt X \to X$ amounts to specifying a $\pi_1(X, x)$-action on $\mathcal M$, this can be seen as an instance of the Cartan--Leray spectral sequence \cite[Proposition 3.6]{heuer-v_lb_rigid}.
		
		For the second claim (about local systems) note that it follows in the same vein (using $H^0(\wt X, K) = K$) that $\mathrm{Rep}_K(\pi_1(X, x))$ is naturally equivalent to the category of $K$-local systems on $X_v$ which are constant on $\wt X$. If $K = \C_p$ one checks easily that this is indeed precisely the category of $\C_p$-local systems with integral model (for one direction use \cite[Corollary 3.21]{MannWerner_LocSys_p-adVB}).
	\end{proof}
	
	\begin{Remark}
		We learnt the idea that pro-finite-\'etale covers of $X$ can be used to study the $p$-adic Hodge theory of $X$ from Bhatt who in \cite{Lectures_Arizona} uses them for abelian varieties of good reduction to prove the Hodge--Tate decompositon in this case.
	\end{Remark}

	%------------------------------------------------------------------------------
	
	\subsection{Higgs bundles} \label{sec:prelim-higgs-bundles}
	Let $X$ be a smooth rigid space over $K$.
	To motivate the definition of Higgs bundles in the $p$-adic setting,
	let us begin by recalling Scholze's perspective on the Hodge--Tate spectral sequence: In \cite{Scholze2012Survey}, Scholze proves that the Leray spectral sequence for $\nu:X_v\to X_{\an}$ applied to the  structure sheaf $\O$ can be interpreted as the Hodge--Tate spectral sequence in case that $X$ is proper. In low degrees, this is in general a left-exact sequence
	
	\begin{equation}\label{HTseq}
		0\to H^1_{\an}(X,\O)\to H^1_{v}(X,\O)\xrightarrow{\HT} H^0_{\an}(X,\Omega^1(-1))\to 0
	\end{equation}
	that is also right-exact if $X$ is proper.
	Here the $(-1)$ is a Tate twist, which in the absence of Galois actions simply means tensoring with the free $\Z_p$-module $\Hom(T_p\mu_{p^\infty},\Z_p)$. Of course since we are working over an algebraically closed field, any choice of a compatible system of $p$-power unit roots induces an isomorphism $\Omega^1(-1) \cong \Omega^1$, but it is more natural not to make such a choice: For example, the Tate twist is important to keep track of the Galois action if $X$ has a model over a local field. 
	
	\begin{Definition}
		To simplify notation, we shall from now on write
		\[ \wt \Omega^1:=\Omega^1(-1).\]
	\end{Definition}
	\begin{Example}
		To see why the twist appears, recall that in the case of abelian varieties, one can define $\HT$ via the morphism
		\[ T_pA^\vee\to H^0(A,\Omega^1)\]
		given by regarding an element of $T_pA^\vee$ as a morphism $\Q_p/\Z_p\to A^\vee[p^\infty]$, dualising, and sending this to the pullback of $dT/T$ on $\mu_{p^\infty}$. Extending $K$-linearly, we can then use the Weil pairing to identify 
		\[T_pA^\vee\otimes K=\Hom(T_pA,\Z_p(1))\otimes K=H^1_v(A,\O)(1).\] Twisting by $(-1)$ gives the map $\HT$ in \cref{HTseq}.
	\end{Example}
	The upshot of this discussion is that in the $p$-adic situation, it is natural to include a Tate twist in the definition of Higgs bundles:
	\begin{Definition}\label{d:Higgs-bundles-2}
		A \emph{Higgs bundle} on $X$ is a pair $(E, \theta)$, where $E$ is an analytic vector bundle on $X$ and $\theta$ is an element 
		\begin{align*}
			\theta \in H^0(X, \IEnd(E) \tensor \wtOm^1)
		\end{align*}
		satisfying the Higgs field condition
		\[ \theta\wedge \theta=0.\]
		Note that we can view $\theta$ as a map $E \to E \tensor \wt\Omega^1$.

		A morphism $(E, \theta) \to (E', \theta')$ of Higgs bundles on $X$ is a map $\varphi:E \to E'$ of vector bundles such that the following diagram commutes:
		\[ \begin{tikzcd}
			E \arrow[d,"\varphi"] \arrow[r,"\theta"] & E\otimes \wtOm^1 \arrow[d,"\varphi\otimes \id"] \\
			E' \arrow[r,"\theta'"] & E'\otimes \wtOm^1
		\end{tikzcd}\]
	\end{Definition}
	
	The category of Higgs bundles is an exact tensor category, where exactness is measured on the underlying vector bundles and the tensor product is defined as
	\begin{align*}
		(E, \theta) \tensor (E', \theta') := (E \tensor E', \theta \tensor \id_E' + \id_E \tensor \theta')
	\end{align*}
	with identity object given by the ``trivial Higgs bundle'' $(\O, 0)$.
	
	\begin{Definition}
		Let $X$ be a smooth rigid variety over $K$.
		\begin{defenum}
			\item A \emph{pro-finite-\'etale Higgs bundle} on $X$ is a Higgs bundle $(E, \theta)$ on $X$ such that $E$ is a pro-finite-\'etale vector bundle.
			
			\item A \emph{unipotent Higgs bundle} on $X$ is a Higgs bundle $(E, \theta)$ on $X$ which can be written as a successive extension (in the category of Higgs bundles) of the trivial Higgs bundle.
		\end{defenum}
	\end{Definition}
	
	\subsection{The $p$-adic Corlette--Simpson correspondence for line bundles} \label{sec:prelim-simpson-for-line-bundles}
	
	The category of $v$-line bundles can be described explicitly in terms of homological algebra. This is possible because the group $\GL_1=\G_m$ is abelian.  We shall now review how one can in particular describe the pro-finite-\'etale line bundles, and the statement of the $p$-adic Corlette--Simpson correspondence in this case.

	The key point is that there is an analog of Scholze's Hodge--Tate spectral sequence when $\O$ is replaced by the sheaf of units $\O^\times$: As explained in \cite[\S2]{heuer-v_lb_rigid}, the $p$-adic logarithm can be used to show that there is for any smooth rigid space $X$ over $K$ a left-exact sequence
	\[ 0\to \Pic_{\an}(X)\to \Pic_{v}(X)\xrightarrow{\HT\log} H^0(X,\wt\Omega^1)\to 0\]
	which is also right-exact if $X$ is proper \cite[Theorem~1.3.2]{heuer-v_lb_rigid}. This implies:
	\begin{Theorem}[{\cite[Theorem~5.2]{heuer-v_lb_rigid}}]\label{rslt:Simpson-linebundles}
		Let $X$ be a smooth proper rigid space over $K$. Then any choice of an exponential function (i.e.\ a continuous splitting of $\log$ on $K$) and of a Hodge--Tate splitting sets up an equivalence of categories
		\[ \{\text{$v$-line bundles on $X$}\}\isomarrow \{ \text{Higgs line bundles on $X$}\}\]
		that induces via \cref{rslt:profet-vb-vs-Reps} an equivalence of categories
		\[\Bigg\{\begin{array}{@{}c@{}l}\text{$1$-dim.\ continuous $K$-linear}\\\text{representations of $\pi_1(X,x)$} \end{array}\Bigg\} \isomarrow \Bigg\{\begin{array}{@{}c@{}l}\text{pro-finite-\'etale}\\\text{Higgs line bundles on $X$}\end{array}\Bigg\}.\]
	\end{Theorem}
	In order to understand the $p$-adic Corlette--Simpson correspondence in rank one, it thus remains to determine when a Higgs line bundle is pro-finite-\'etale. This is done as follows:
	\begin{Theorem}[{\cite[Theorem~5.1]{heuer-diamantine-Picard}}]
		Assume that the rigid Picard functor of $X$ is represented by a rigid group variety $G$. Then a line bundle $L$ on $X$ is pro-finite-\'etale if and only if its associated point $L\in G(K)$ is topological torsion, i.e.\ there is $N\in\N$ such that \[L^{Np^n}\to 1 \quad \text{for}\quad n\to \infty.\]
	\end{Theorem}
	In the case of abeloid varieties, it is known that the Picard functor is always representable by a group whose identity component is the dual abeloid $A^\vee$ \cite[\S6]{BL-Degenerating-AV}, and whose N\'eron--Severi group is torsionfree. Consequently, in this case, the pro-finite-\'etale line bundles are precisely the ones represented by the topological torsion subgroup
	\[ A^{\vee}(K)^{\tt}=\Hom_{\cts}(\wh{\Z},A^\vee(K))\subseteq A^\vee(K).\]
	In summary, we have in this case a canonical short exact sequence
	\[ 0\to A^{\vee}(K)^{\tt}\to \Hom_{\cts}(TA,K^\times)\xrightarrow{\HT\log} H^0(A,\Omega^1(-1))\to 0.\]
	Here the first map can be interpreted as the Weil pairing \cite[\S4]{DeningerWerner-lb_and_p-adic-characters} \cite[\S5.2]{heuer-diamantine-Picard}.
	\section{Pro-finite-\'etale vector bundles on abeloids}
	
	As before, let $K$ be a complete algebraically closed field extension of $\Q_p$. In this section, we show that on an abeloid variety $A$ over $K$ any pro-finite-\'etale $v$-vector bundle can be built out of pro-finite-\'etale line bundles, which are well-understood by the $p$-adic Corlette--Simpson correspondence in rank 1 of \cite[\S5]{heuer-v_lb_rigid}, and unipotent vector bundles. We begin by studying the latter.
	\subsection{Unipotent $v$-vector bundles}
	In this subsection, we can again work in greater generality: We assume throughout that $X$ is a smooth rigid space over $K$.
	\begin{Definition}
		\begin{defenum}
			\item A $v$-vector bundle $E$ on $X$ is called $v$-\emph{unipotent} if it is a successive extension of trivial $v$-line bundles on $X_v$.
			\item An analytic vector bundle $E$ on $X$ is \emph{analytically unipotent} if $E$ is a successive extension of trivial line bundles on $X_\an$.
		\end{defenum}
	\end{Definition}
	
	Of course any analytic vector bundle that is analytically unipotent is automatically v-unipotent. We will now show that the converse is true.
	
	\begin{Lemma} \label{rslt:unipotent-plus-analytic-implies-analytically-unipotent}
		Let $E$ be an analytic vector bundle on $X$. Then $E$ is $v$-unipotent if and only if $E$ is analytically unipotent.
	\end{Lemma}
	We will therefore in the following just speak of ``unipotent'' analytic or $v$-vector bundles.
	\begin{proof} We only have to show that every $v$-unipotent bundle $E$ is also analytically unipotent. 
		We prove this by induction on the rank $n$ of $E$. Since $E$ is unipotent in the $v$-topology, we can find a $v$-topological extension
		\begin{align*}
			0 \to \O \xto{\varphi} E \to E' \to 0,
		\end{align*}
		where $E'$ is an unipotent $v$-bundle of smaller rank. By Lemma~\ref{l:an-to-v-vb-fully-faithful}, the morphism $\varphi\colon \O \to E$ exists already in the analytic topology, where it is clearly still injective. Let now $E'' := E / \O$ in the analytic topology. Then $E''$ is coherent on $X_\an$. By right-exactness of $\nu^*$ we have $\nu^* E''= E'$. By \cref{rslt:v-loc-free-implies-analytic-loc-free} below it follows that $E''$ is analytically locally free, necessarily of rank $n-1$. We win by induction.
	\end{proof}
	
	\begin{Lemma} \label{rslt:v-loc-free-implies-analytic-loc-free}
		Let $X$ be a smooth rigid space over $K$ and let $F$ be a coherent sheaf on $X_\an$. Then $F$ is locally free on $X_\an$ if and only if its pullback to $X_v$ is $v$-locally free.
	\end{Lemma}
	\begin{proof}
		We only need to show the ``if'' direction, so let us assume that the pullback $\nu^* F$ of $F$ to $X_v$ is $v$-locally free. We need to show that then $F$ is locally free on $X_\an$. This statement is local on $X$, so we can assume that $X = \Spa(A, A^+)$ is affinoid and admits an étale map to a torus $T$ that is a composition of rational localizations and finite étale maps. Let $T_\infty \to T$ be the affinoid pro-\'etale toric tower obtained by adjoining all $p$-power roots of the coordinates. By  \cite[Lemma~4.5]{Scholze_p-adicHodgeForRigid}, base-changing to $X$ yields a pro-finite-étale cover $X_\infty \to X$ of $X$ by an affinoid perfectoid space $X_\infty = \Spa(A_\infty, A^+_\infty)$.
		
		By \cite[Theorem 2.3.3]{relative-p-adic-hodge-2}, $F$ corresponds to a finite $A$-module $M$. Since $\nu^* F$ is $v$-locally free by assumption, \cite[Theorem 3.5.8]{relative-p-adic-hodge-2} (see also \cite[Lemma 17.1.8]{SW_Berkeley}) implies that $(\nu^* F)|_{X_\infty}$ is analytic locally free and corresponds to a finite projective module $M_\infty$. Clearly $M_\infty = M \tensor_A A_\infty$. We now use that the morphism $A\to A_\infty$ is split in the category of $A$-modules: In our situation, this holds because one can pull back splittings for the toric tower using the explicit descriptions in \cite[Lemma~4.5]{Scholze_p-adicHodgeForRigid}, but we mention that such splittings exist more generally, see \cite[Lemma 7.6]{HK_sheafiness}. In particular, $A\to A_\infty$ is universally injective, so finite projective modules descend along $A \to A_\infty$ (see e.g. \cite[Theorem 08XD]{StacksProject}). It follows that $M$ is finite projective and hence that $F$ is locally free, as desired.
	\end{proof}
	
	The following analyticity criterion will be helpful when working with analytic vector bundles as $v$-vector bundles:
	\begin{Lemma}[{\cite[Corollary~3.5]{heuer-v_lb_rigid}}]\label{rslt:analyticity-of-sub-LB-of-VB} Let $X$ be a connected smooth rigid space.
		Let $V$ be an analytic vector bundle and let $L$ be a $v$-line bundle on $X$. Suppose that there is a non-trivial map $L\to V$. Then $L$ is analytic.
	\end{Lemma}
	
	\subsection{Classification of pro-finite-\'etale vector bundles on abeloids}
	We now come to the promised structure result for pro-finite-\'etale vector bundles on abeloid varieties. This is an analog of the following classical result.	Let $A$ be an abelian variety over an algebraically closed field $F$. Then a vector bundle $E$ on $A$ is called homogeneous if it satisfies $x^{\ast}E=E$ for all $x\in A(F)$. We make analogous definitions for complex tori over $F=\C$, as well as abeloids over $K$. By a theorem of Matsushima, Morimoto \cite{Morimoto}\cite{Matsushima-Morimoto} (in the setting of complex tori) and Miyanishi \cite{Miyanishi}, Mukai \cite[Theorem 4.17]{Mukai_semi-homog-VB_on_AV} (in the setting of abelian varieties), a vector bundle $E$ on $A$ is homogeneous if and only if it decomposes as a direct sum
	\[ E = \bigoplus_{i=1}^n (U_i\tensor L_i),\]
	where each $U_i$ is unipotent and each $L_i$ is a homogeneous line bundle.
	
	We now prove an analogous result for pro-finite-\'etale vector bundles on abeloids:
	\begin{Theorem}
		\label{rslt:vb-on-A-are-unipotent-times-lb}
		Let $A$ be an abeloid variety over $K$ and let $E$ be a $v$-vector bundle on $A$.
		\begin{propenum}
			\item $E$ is pro-finite-\'etale if and only if it decomposes as a direct sum
			\begin{align*}
				E = \bigoplus_{i=1}^n (U_i\tensor L_i),
			\end{align*}
			where each $U_i$ is a unipotent $v$-vector bundle and each $L_i$ is a pro-finite-\'etale $v$-line bundle on $A$. For each factor, both $L_i$ and $U_i$ are unique up to isomorphism.
			
			\item If $E$ is moreover analytic, then
			each $U_i$ is a unipotent analytic vector bundle and each $L_i$ is a pro-finite-\'etale analytic line bundle on $A$.
		\end{propenum}
	\end{Theorem}
	\begin{Remark}\label{rem:homogeneous}
		It is easy to see that any unipotent vector bundle is homogeneous.
		The Theorem therefore implies that any pro-finite-\'etale vector bundle on an abeloid variety $A$ is homogeneous. The converse is true for abelian varieties (and probably also abeloids) over $\C_p$, but not over more general fields: already in the case of line bundles one has to instead impose the condition that $E$ corresponds to a topological torsion point of $A^\vee(K)^{\tt}$ (see \cref{sec:prelim-simpson-for-line-bundles}), which is stronger than being homogeneous.  
		
	\end{Remark}
	
	Most of the remaining part of this section is devoted to presenting a proof of \cref{rslt:vb-on-A-are-unipotent-times-lb}. At the end (see \cref{rslt:decomposition-of-profet-vb-category}) we also give a slightly different categorical interpretation of this result.
	
	\begin{Lemma} \label{rslt:H^1_cts(whZ-K)}
		Let $\chi:\Z^n_p\to K^\times$ be a continuous character and consider $K$ endowed with the $\Z_p^n$-action via $\chi$. Then for any $d\geq 0$,
		\begin{align*}
			H^d_{\cts}(\Z_p^n,K)=\begin{cases}\wedge^dK^n\quad& \text{ for }\chi=1,\\0& \text{ for }\chi\neq 1.\end{cases}
		\end{align*}
	\end{Lemma}
	\begin{proof}
		Let $\gamma_1,\dots,\gamma_n$ be the images of the standard basis of $\Z_p^n$ under $\chi$.
		We recall that in general, $R\Gamma_{\cts}(\Z_p^n,K)$ is computed by the Koszul complex
		\[K(\gamma_1,\dots,\gamma_n):= \Big[K\to K^n\to \dots\to \wedge^dK^n\to \dots K^n\to K\Big].\]
		This follows from the argument in \cite[Proof of Lemma~5.5]{Scholze_p-adicHodgeForRigid} by first treating $\O_K=\varprojlim \O_K/p^n$  and then inverting $p$.
		
		In the case of $n = 1$, this complex is simply the multiplication $(\gamma_1-1):K\to K$
		and we deduce that
		\begin{align*}
			H^1_{\cts}(\Z_p,K) = K/(\gamma_1-1)K.
		\end{align*}
		This is trivial unless $\gamma_1$ acts trivially, thus $H^1_{\cts}(\Z_p,K)$ is as described. The statement for $d=0$ is clear, and cohomology vanishes for $d>1$.
		
		To deduce the general case, we first note  that
		\[ K(\gamma_1,\dots,\gamma_n)=K(\gamma_1-1)\otimes \dots\otimes K(\gamma_n-1).\]
		Inductively, the double complex spectral sequence now shows that
		\[ H^d_{\cts}(\Z_p^n,K)=\bigoplus_{i_1+\dots+i_n=d}H^{i_1}_{\cts}(\Z_p,K)\otimes \dots \otimes H^{i_n}_{\cts}(\Z_p,K).\]
		We deduce from the case of $n=1$ that this is zero unless each $\gamma_i$ acts trivially, in which case it is $\wedge^dK^n$, as desired.
	\end{proof}
	
	For the following results, we need an intermediate pro-finite-\'etale pro-$p$-cover defined as follows. 
	\begin{Definition}\label{d:p-adic-cover}
		Let $A$ be an abeloid over $K$. Then we define the $p$-adic universal cover of $A$ as
		\begin{align*}
			\wt A^p := \varprojlim_{[p]}A\to A.
		\end{align*}
	\end{Definition}
	Similarly as for $\wt A$, we have (we caution that our $\wt A^p$ is what is denoted by $\wt A$ in \cite{heuer-Picard-good-reduction}):
	
	\begin{Proposition}[{\cite[Theorem~1]{perfectoid-covers-Arizona}, \cite[Proposition~4.2]{heuer-Picard-good-reduction}}]\label{rslt:p-adic-cover-O-acyclic}
		The space $\wt A^{p}$ is perfectoid and satisfies for $i\geq 1$:
		\[H^i_v(\wt A^p,\Z_p)=0,\quad H^i_v(\wt A^p,\O^+)\aeq 0.\]
	\end{Proposition}
	Next, we prove an analog of the classical lemma in the theory of abelian varieties that $\mathrm R\Gamma(A,L)=0$ for any non-trivial homogeneous line bundle $L$ on $A$, see \cite[Chapter 8, (vii)]{Mumford}.
	\begin{Lemma} \label{rslt:H^1_v(A-L)}
		Let $L\neq \O$ be a pro-finite-\'etale $v$-line bundle on an abeloid variety $A$ over $K$. Then
		\begin{align*}
			H^i_v(A,L)=0 \quad \text{for all $i\geq 0$}.
		\end{align*}
	\end{Lemma}
	\begin{proof}
		The $v$-line bundle $L$ corresponds to a non-trivial character $\pi_1(A,0)\to K^\times$ by \cref{rslt:profet-vb-vs-Reps}. We first consider descent along the prime-to-$p$-torsor $\wt A\to \wt A^{p}$.
		Since the action of $T^pA$ factors through a finite quotient $A[N]$ for some $N$ coprime to $p$, the bundle $L$ already becomes trivial on the finite cover $[N]:\wt A^{p}\to \wt A^{p}$. Using that $\O(\wt A^p)=H^0_{\cts}(T^pA,\O(\wt A))=K$ by \cref{rslt:properties-of-wt-X}, it follows by the Cartan--Leray spectral sequence (see \cite[Corollary~2.9]{heuer-v_lb_rigid}) that for all $i\geq 0$,
		\[H^i_v(\wt A^p,L) = H^i_{\cts}(A[N], K)\]
		which vanishes if $i\geq 1$ since $K$ is uniquely divisible \cite[Proposition (1.6.2)]{NeuSchWin}. For ${i=0}$, it is clear that $ H^0_{\cts}(A[N], K)=0$ unless $A[N]$ acts trivially, in which case $H^0_v(\wt A^p,L)=K$.
		
		We now consider the Cartan--Leray sequence for the $T_pA$-torsor $\wt A^p\to A$ (also compare to \cite[Propositions 3.5 and 3.7.(iii)]{Scholze_p-adicHodgeForRigid}). The above implies that this is given by
		\begin{align*}
			H^i_v(A,L) = H^i_{\cts}(T_pA, L(\wt A^{p})).
		\end{align*}
		If $L(\wt A^{p})=0$, this vanishes. If $L(\wt A^{p})=K$, the group is computed by
		\cref{rslt:H^1_cts(whZ-K)}.
	\end{proof}
	\begin{Remark}
		This argument shows that \cref{rslt:H^1_cts(whZ-K)} also holds for $\Z_p$ replaced by $\wh \Z$. On the group cohomological side, this can be seen by replacing the Cartan--Leray sequence by Hochschild--Serre.
		
		Alternatively, one could prove \cref{rslt:H^1_v(A-L)} like in the classical case; for this one first needs to establish a K\"unneth-formula for $v$-vector bundles.
		
	\end{Remark}
	\begin{Corollary} \label{rslt:Ext^1(L_1-L_2)}
		Let $L_1$, $L_2$ be two pro-finite-\'etale $v$-line bundles on an abeloid variety $A$. Then
		\begin{align*}
			\Ext^1_v(L_1, L_2) = \begin{cases}H^1_v(A, \O) \quad& \text{if $L_1=L_2$},\\ 0 & \text{if $L_1 \neq L_2$}.\end{cases}
		\end{align*}
	\end{Corollary}
	\begin{proof}
		We have 
		\begin{align*}
			\Ext^1_v(L_1,L_2)=\Ext^1_v(\O, L_2 \tensor L_1^{-1}) = H^1_v(A, L_2 \tensor  L_1^{-1}).
		\end{align*}
		The statement now follows from \cref{rslt:H^1_v(A-L)}.
	\end{proof}
	
	\begin{proof}[Proof of \cref{rslt:vb-on-A-are-unipotent-times-lb}]
		By \cref{rslt:properties-of-wt-X} there are no non-trivial extensions of $\O$ by itself on $\wt A$, hence every unipotent $v$-vector bundle on $A$ is pro-finite-\'etale. So in both (i) and (ii) we only need to show the ``only if'' part.
		
		Assume that $E$ is pro-finite-\'etale. Then by Theorem~\ref{rslt:profet-vb-vs-Reps}, $E$ corresponds to a representation $\pi_1(A,0)\to \GL(W)$ where $W:=E(\wt X)$. Since $\pi_1(A,0)$ is abelian and topologically finitely generated, there is a simultaneous eigenvector, i.e.\ the representation is upper triangular. By exactness of the equivalence of the theorem, it follows inductively that $E$ is a successive extension of line bundles. By \cref{rslt:Ext^1(L_1-L_2)}, any such extension decomposes into a direct sum of successive extensions $E_i$ of some line bundle $L_i$ by itself. Then $U_i=V_i\tensor L_i^{-1}$ is unipotent, proving (i).
		
		To prove (ii), let us now assume that $E$ is analytic and pro-finite-\'etale. By (i) we can find a decomposition $E = \bigoplus_{i=1}^n (U_i \tensor L_i)$ of $E$ as a $v$-bundle. Then each $W:=U_i \tensor L_i\hookrightarrow E$ is itself an analytic vector bundle since it is a direct factor in $E$: Indeed, we can write it as the kernel of the morphism $E\to E/W\to E$. By fully faithfulness, this is a morphism of analytic vector bundles, and thus its kernel is an analytic vector bundle by Lemma~\ref{rslt:v-loc-free-implies-analytic-loc-free}.
		
		By \cref{rslt:analyticity-of-sub-LB-of-VB}, this implies that $L_i$ is analytic. Consequently,  $U_i=(U_i\tensor L_i)\otimes L_i^{-1}$ is analytic. It is then also analytically unipotent by \cref{rslt:unipotent-plus-analytic-implies-analytically-unipotent}.
	\end{proof}
	
	The following lemma shows that we can in practice restrict attention to unipotent vector bundles that vanish on the  pro-finite-\'etale pro-$p$-cover $\wt A^p$:
	
	\begin{Lemma}\label{rslt:unipotent-trivial-on-wtA}
		Let $E$ be a unipotent $v$-vector bundle on an abeloid variety $A$ over $K$. Then $E$ becomes trivial on $\wt A^p$. In particular, the corresponding representation from Theorem \ref{rslt:profet-vb-vs-Reps} factors through 
		\[T_pA\to \GL_n(K).\]
	\end{Lemma}
	\begin{proof}
		A $v$-vector bundle is unipotent if and only if it is pro-finite-\'etale and the associated representation $\pi_1(A, 0) = TA \to \GL_n(K)$ factors over a conjugate of the subgroup $U\subseteq \GL_n(K)$ of upper triangular unipotent matrices. Since $U$ is a pro-$p$-group, any morphism $TA\to U$ factors through $T_pA\to U$. The claim now follows by looking at the associated descent data (equivalently, by comparing the Cartan--Leray sequences for $\wt A$ and $\wt A^p$).
	\end{proof}
	
	As a  corollary we derive the following decomposition of the category of all pro-finite-\'etale vector bundles:
	
	\begin{Corollary} \label{rslt:decomposition-of-profet-vb-category}
		Let $A$ be an abeloid variety over $K$ and let $\Pic^{\pf}_v(A)$ denote the set of isomorphism classes of pro-finite-\'etale $v$-line bundles on $A$. Then there is a natural equivalence of categories
		\begin{align*}
			\{ \text{pro-finite-\'etale $v$-bundles on $A$} \} \ &\cong \ \bigoplus_{[L] \in \Pic^\pf_v(A)} \{ \text{unipotent $v$-bundles on $A$} \}\\
			\bigoplus_{L} L\tensor U_L \ & \mapsfrom \ (U_{L})_{L}
		\end{align*}
		Here for each class $[L]$ in $\Pic_v^\pf(A)$, we have a canonical and natural representative $L$ defined as follows: Write $\Pic_v^\pf(A)=\Hom_{\cts}(TA,K^\times)$ and associate to each character $\rho$ the line bundle $V_{\rho}$ defined in \cref{rslt:profet-vb-vs-Reps}. 
	\end{Corollary}
	\begin{proof}
		By \cref{rslt:vb-on-A-are-unipotent-times-lb} the functor from right to left is essentially surjective, so we only need to check full faithfulness. This boils down to showing the following: Given two unipotent $v$-bundles $U_1$, $U_2$ and two pro-finite-\'etale $v$-line  bundles $L_1 \ne L_2$ we have
		\begin{align*}
			\Hom(L_1 \tensor U_1, L_2 \tensor U_2) = 0.
		\end{align*}
		Tensoring with $L_1^{-1}$ reduces to the case that $L_1 = \O$ is trivial. Note that $L_2 \tensor U_2$ is a successive extension of copies of $L_2$, so inductively \cref{rslt:H^1_v(A-L)} shows $H^i_v(A, L_2 \tensor U_2) = 0$. Now suppose we are given a morphism $f\colon U_1 \to L_2 \tensor U_2$ and write $U_1$ as an extension $\O\to U_1\to U_1'$ where $U_1'$ is unipotent of smaller rank. Then the composition \[\O \injto U_1 \xto{f} L_2 \tensor U_2\]
		can be seen as an element of $H^0(A, L_2 \tensor U_2)$ and hence vanishes. It follows that $f$ factors through $U_1'$. Inductively, we obtain $f = 0$, as desired.
	\end{proof}

	%------------------------------------------------------------------------------
	\section{The diamantine universal vector extension}
	As before, let $K$ be a complete algebraically closed field extension of $\Q_p$. We have seen in \cref{rslt:profet-vb-vs-Reps} that pro-finite-\'etale $v$-vector bundles on an abeloid variety $A$ over $K$ are equivalent to finite dimensional continuous representations of the fundamental group $TA = \pi_1(A, 0)$ over $K$. This raises the natural question:
	\begin{question}\label{q:when-is-rep-an}
		Given a continuous representation $TA\to \GL_n(K)$, how can we tell whether the associated pro-finite-\'etale $v$-vector bundle is analytic?
	\end{question}

	In this section, we will answer the  question above  with the help of universal vector extensions. Classically, Rosenlicht \cite{rosenlicht-universal} has shown that every abelian variety $A$ has a universal extension among all extensions of $A$ by vector groups. Here the vector group in the universal vector extension is the tangent space of the dual abelian variety. In this section we will generalise these results from the algebraic to the analytic setting, i.e.\ to the case of abeloid varieties. In a second generalisation direction, we prove the existence of universal extensions also for the $v$-topological analog of vector extensions (see \cref{d:vectorext}). 
	We then answer Question~\ref{q:when-is-rep-an} by generalising results of Brion  (\cite{Brion_homogVB_AV}, see also \cite[Theorem 6.4.1, Proposition 6.4.4]{BPU_Chennai}) relating universal vector extensions to unipotent representations.
	
	\subsection{Universal vector extensions of abeloids}
	
	In the following we will work with the ``big analytic site'' $\LSD_{K,\an}$ (where covers are the open covers) and the ``big $v$-site'' $\LSD_{K,v}$ of locally spatial diamonds over $K$ (see \cref{sec:prelim-diamonds}).
	
	\begin{Definition}\label{d:vectorext}
		\begin{defenum}
			\item A \emph{vector group over $K$} is a (rigid) group variety $V$ over $K$ which is isomorphic to a finite product of copies of $\mathbb G_a$. To any finite-dimensional $K$-vector space $W$ we can associate the vector group $W \tensor \mathbb G_a$. As usual, via diamondification we consider $V$ as an object of $\LSD_{K,v}$.
			
			\item 
			Let $A$ be an abeloid variety over $K$ and let $\tau \in \{ \an, v \}$. A \emph{$\tau$-vector extension of $A$} is an exact sequence  of sheaves of abelian groups on $\LSD_{K,\tau}$
			\begin{align*}
				0 \to V \to E \to A \to 0
			\end{align*}
			where $V$ is a vector group over $K$.
		\end{defenum}
	\end{Definition}
	
	\begin{Remark}
		It is clear that the isomorphism classes of $\tau$-vector extensions of an abeloid variety $A$ by a given vector group $V$ are precisely the elements of $\Ext^1_\tau(A, V)$. Note that for $\tau = \an$, the groups $\Ext^n_\an(A, V)$ can equivalently be computed either on the site $\LSD_{K,\an}$ or on $A_\an$ (or on the big analytic site of rigid $K$-varieties). Similarly, $H^n_\an(A, V)$ can equivalently be computed in $\LSD_{K,\an}$ or $A_\an$. Thus the choice of working with $\LSD_{K,\an}$ instead of $A_\an$ is really a matter of taste -- the reason we choose to work with $\LSD_{K,\an}$ is that it allows for a uniform treatment of the analytic and $v$-topology.
	\end{Remark}
	
	By the following lemma, $\tau$-vector extensions of $A$ correspond to $\tau$-$V$-torsors over $A$.
	
	\begin{Lemma} \label{rslt:Ext-A-Ga-is-H1-of-O}
		Let $A$ be an abeloid variety over $K$ and let $\tau \in \{ \an, v \}$. Let $V$ be a vector group. Then there is a natural isomorphism of abelian groups
		\begin{align*}
			\Ext^1_\tau(A, V) = H^1_\tau(A, V).
		\end{align*}
		In particular, every analytic vector extension of $A$ is also a $v$-vector extension and moreover is representable by a rigid group variety over $K$.
	\end{Lemma}
	This follows from the following very general statement:
	\begin{Lemma}\label{rslt:BD-in-deg-1}
		Let $\mathcal C$ be any site  and let $G$ and $H$ be abelian sheaves on $\mathcal C$ such that any morphism $G^n\to H$ for $n\in \N$ is constant. Then
		\[ \Ext^1(G, H) = \ker\big(H^1(G, H)\xrightarrow{\delta} H^1(G\times G, H)\big)\]
		where $\delta=m^{\ast}-\pi_1^{\ast}-\pi_2^{\ast}$.
	\end{Lemma}
	\begin{proof}
		This follows from the Breen--Deligne resolution \cite[\S 2.1.5]{BBM}  (see also \cite[Theorem 4.5]{condensed-mathematics}): This is a functorial resolution of $G$ in $\mathcal C$ of the form
		\begin{align*}
			\dots \bigoplus_{k=1}^{n_i} \Z[G^{r_{i,k}}] \to \dots \to \Z[G^3] \oplus \Z[G^2] \to \Z[G^2] \xrightarrow{[m]-[\pi_1]-[\pi_2]} \Z[G] \to G \to 0.
		\end{align*}
		From this resolution we directly obtain a spectral sequence
		\begin{align*}
			E_1^{i,j} = \prod_{k=1}^{n_i} H^j(G^{r_{i,k}}, H) \Longrightarrow \Ext^{i+j}(G,  H).
		\end{align*}
		By assumption, the terms $E_1^{i,0}$ agree with those computing $\Ext(0,H)=0$,  and thus the terms $E_2^{i,0}$ vanish. It follows from the $5$-term exact sequence that 
		\begin{align*}
			\Ext^1(G, H) = E_2^{0,1} = \ker(H^1(G, H) \xto{\delta} H^1(G\times G, H)).
		\end{align*}
	\end{proof}
	\begin{proof}[Proof of \cref{rslt:Ext-A-Ga-is-H1-of-O}]
		We apply \cref{rslt:BD-in-deg-1} to the site $\LSD_{K,\tau}$. It then remains to show that $\delta$ vanishes. Since $H^1_\an(A^r, \O) \subset H^1_v(A^r, \O)$ for $r=1,2$, it suffices to show this vanishing in the case $\tau = v$. Then by Cartan--Leray, $H^1_v(A, \O)$ is continuous $\pi_1(A, 0)$-cohomology of the trivial representation, hence $H^1_v(A, \O) = \Hom_{\cts}(T_p A, K)$. Similarly $H^1_v(A^2, \O) = \Hom_{\cts}(T_p A \times T_p A, K)$ and $\delta$ is the map
		\begin{align*}
			\Hom_{\cts}(T_p A, K) &\to \Hom_{\cts}(T_p A \times T_p A, K),\\
			f &\mapsto \varphi(f) = [(x, y) \mapsto f(x+y) - f(x) - f(y)],
		\end{align*}
		which is evidently the zero map.
	\end{proof}
	In the $v$-topology, the description of the last lemma can be made even more explicit:
	\begin{Lemma}\label{rslt:H1-as-Homs}
		Let $V$ be any unipotent $v$-sheaf, for example a vector group. Then we have
		\[H^1_v(A,V)=\Hom_{\cts}(T_pA,V(K)).\]
	\end{Lemma}
	\begin{proof}
		The Cartan--Leray sequence of the cover $\wt A^{p}
		\to A$ induces an exact sequence
		\[0\to H^1_{\cts}(T_pA,V(K))\to H^1_v(A,V)\to H^1_v(\wt A^{p},V).\]
		The first term are precisely the continuous homomorphisms. The last term vanishes, this follows inductively from $H^1_v(\wt A^{p},\O)=0$ by \cref{rslt:p-adic-cover-O-acyclic}.
	\end{proof}
	
	\begin{Lemma}\label{rslt:Hom(G_a)}
		For the rigid group $\G_a$, we have
		$\Hom(\G_a,\G_a)=K.$ In particular, sending
		\[ V\mapsto V\otimes_{K}\G_a\]
		defines an equivalence of categories from finite dimensional $K$-vector spaces to vector groups.
	\end{Lemma}
	\begin{proof}
		Taking global sections, $\Hom(\G_a,\G_a)$ injects into the endomorphisms of the formal group $K[[X]]$ with its additive group law. Here the statement is clear.
	\end{proof}
	
	Now we are ready to generalise the classical result that there exists a universal vector extension in the Zariski topology on abelian varieties  \cite{rosenlicht-universal} to the rigid analytic setting of abeloids. Moreover, we prove that this also holds in the $v$-topology:
	\begin{Definition}
		For any $K$-vector space $W$,  we denote by $W^\ast=\Hom_K(W,K)$ its $K$-dual.
	\end{Definition}
	\begin{Theorem} \label{rslt:universal-vector-extension-exists}
		Let $A$ be an abeloid variety over $K$ and let $\tau \in \{ \an, v \}$. Then there is a universal $\tau$-vector extension
		\begin{align*}
			0 \to H^1_\tau(A, \O)^* \tensor \mathbb G_a \to E_\tau A \to A \to 0,
		\end{align*}
		i.e.\ $E_\tau A$ is a $\tau$-vector extension  and for every vector extension $0 \to V \to E \to A \to 0$ of $A$ there is a unique map $f\colon H^1_\tau(A, \O)^* \to V$ of group varieties such that $E$ is the pushout of $E_\tau A$ along $f$.
	\end{Theorem}
	\begin{proof}
		For any vector group $V$ over $K$ we have, by \cref{rslt:Ext-A-Ga-is-H1-of-O},
		\begin{align*}
			\Ext^1_\tau(A, V) = H^1_\tau(A, V) = H^1_\tau(A, \O) \tensor_K V(K) = \Hom_K(H^1_\tau(A, \O)^*, V(K)).
		\end{align*}
		In other words, the $\tau$-vector extensions of $A$ by $V$ are in one-to-one correspondence to $K$-linear maps $f\colon H^1_\tau(A, \O)^* \to V(K)$. Thus, letting $V = H^1_\tau(A, \O)^* \tensor \mathbb G_a$ and $f = \id$, it is immediate from \cref{rslt:Hom(G_a)} that we obtain a vector extension $E_\tau A$ which is universal.
	\end{proof}
	
	\subsection{Unipotent representations from vector extensions}
	Having established the theory of vector extensions of abeloids, we can now generalise a result of Brion for abelian varieties (\cite{Brion_homogVB_AV}, section 3.3).
	
	\begin{Definition}
		Let $V$ be a vector group over $K$. By a representation of $V$ on a (finite dimensional) $K$-vector space $W$ we mean a homomorphism of rigid group varieties \[\rho:V\to \GL(W).\]
		We say that this is algebraic if it arises from a morphism of group schemes $V\to \GL(W)$ via analytification. (We will see below that this is always the case.)
	\end{Definition}
	Given any $\tau$-$V$-vector extension $V\to E\to A$ (where $\tau$ denotes either analytic or $v$-topology), and any representation $\rho:V\to \GL(W)$, the push-out 
	\[ E\times^{\rho}\GL(W)\to E\]
	of $E$ along $\rho$ is a $\tau$-vector bundle on $A$. The following theorem says that the vector bundles arising in this way are precisely the unipotent ones.
	This is the main result of this section: 
	\begin{Theorem} \label{rslt:characterize-tau-unipotent-VB-via-algebraic-reps}
		Let $A$ be an abeloid variety over $K$ and let $\tau \in \{ \an, v \}$. Then the universal $\tau$-vector extension $E_{\tau}A$ defines an equivalence of categories
		\begin{alignat*}{2}
			\{\text{representations of $H^1_\tau(A, \O)^* \tensor \mathbb G_a$} \}&\isomarrow&&\{ \text{unipotent $\tau$-vector bundles on A}\}\\
			\rho:H^1_\tau(A, \O)^* \tensor \mathbb G_a\to \GL(W)\quad\quad &\mapsto && \quad\quad E_{\tau}A\times^{\rho} \GL(W)
		\end{alignat*}
	\end{Theorem}
	
	For the proof, we start with the following two easy results about algebraic representations of vector groups in general. We first recall the definition of the matrix logarithm and exponential. These will also be important later on.
	
	\begin{Lemma} \label{rslt:matrix-log-exp}
		Let $R$ be a $\Q$-algebra, $W$ a finite free $R$-module and let $U(W)\subseteq \GL_R(W)$ and $N(W)\subseteq \GL_R(W)$ denote the sets of unipotent and nilpotent $R$-linear automorphisms of $W$, respectively. Then there is a natural bijection
		\begin{align*}
			U(W) &\isotofrom N(W)\\
			u & \mapsto \log(u) := \sum_{k=1}^\infty \frac{(-1)^{k+1}}k (u - 1)^k\\
			\exp(n) := \sum_{k=0}^\infty \frac1{k!} n^k & \mapsfrom n
		\end{align*}
		Moreover, we have:
		\begin{lemenum}
			\item If $u, u' \in U(W)$ commute then $\log(u)$ and $\log(u')$ commute and 
			\[\log(uu') = \log(u) + \log(u').\]
			\item If $n, n' \in N(W)$ commute then $\exp(n)$ and $\exp(n')$ commute and 
			\[\exp(n + n') = \exp(n) \exp(n').\]
			\item Let $u\in U(W)$ and $n\in N(W)$, then the following are equivalent:
			\[\text{$u$ and $n$ commute} \Leftrightarrow 
			\text{$u$ and $\exp (n)$ commute} \Leftrightarrow
			\text{$\log(u)$ and $n$ commute}.\]
			
		\end{lemenum}
	\end{Lemma}
	\begin{proof}
		The involved power series are all finite by definition of unipotent and nilpotent automorphisms. The bijection and (i) and (ii) therefore follow from the analogous statements about formal power series in $R[[X]]$, respectively $R[[X,X']]$.
		
		To see (iii), it is clear from (i) and (ii) that the last two statements are equivalent. For the first equivalence, the implication $\Rightarrow$ is also clear. To see $\Leftarrow$, observe that $unu^{-1}$ is also a nilpotent matrix, and if $u\exp(n)=\exp(n)u$ then
		\[\exp(n)=u\exp(n)u^{-1}=\exp(unu^{-1}).\]
		It thus follows from injectivity of $\exp$ that $n$ and $unu^{-1}$ coincide.
	\end{proof}

	\begin{Lemma} \label{rslt:algebraic-reps-of-vector-groups-equiv-unip-reps-of-Z-p-mod}
		Let $K$ be any non-archimedean field over $\Q_p$ and $T$ a finite free $\Z_p$-module. Consider $V:=T\otimes_{\Z_p}\G_a$ as a vector group. Then any rigid analytic representation
		\[ T\otimes_{\Z_p}\G_a\to \GL_n\]
		is unipotent and factors through some vector-subgroup of $\GL_n$. In particular, the following categories are equivalent:
		\begin{enumerate}[(a)]
			\item Algebraic representations of $V$ on finite-dimensional $K$-vector spaces.
			\item Rigid analytic representations of  $V$ on finite-dimensional $K$-vector spaces.
			\item Continuous unipotent representations of $T$ on finite-dimensional $K$-vector spaces.
			\item Pairs $(W, (u_i)_i)$, where $W$ is a finite-dimensional $K$-vector space and $(u_i)_i$ is an $\rk T$-tuple of pairwise commuting unipotent automorphisms of $W$.
		\end{enumerate}
		The equivalence (a) $\isomarrow$ (b) is analytification,  (b) $\isomarrow$ (c) is given by $\rho \mapsto \rho(K)|_T$ and the equivalence (c) $\isomarrow$ (d) is given by $\rho' \mapsto (\rho'(e_i))_i$, where $(e_i)_i$ is a fixed basis of $T$.
	\end{Lemma}
	\begin{proof}
		We start by proving that any rigid analytic representation $\rho:T\otimes_{\Z_p}\G_a\to \GL_n$ is unipotent, and in particular (b)$\to $(c) is well-defined. To see this, consider the restriction $\rho_{|T}$. We may after conjugating $\rho$ by matrices in $\GL_n(K)$ assume that $\rho_{|T}$ is upper triangular. 
		
		We claim that then already $\rho$ is upper triangular:
		To see this, choose a basis $T\cong \Z_p^r$, then we observe that the matrix entries define rigid analytic functions $\G_a^r\to \G_a$. By the rigid analytic Identity Theorem (see e.g.\ \cite[Corollary~5.1.4.5]{BGR}), the subset $\Z_p^r\subseteq \B^r$ of the closed unit ball $\B^r$ is Zariski-dense, and thus also $\Z_p^r\subseteq \G_a^r$ is Zariski-dense, i.e.\ a rational function is uniquely determined on its values on $\Z_p^r\subseteq \G_a^r$. This shows that the subdiagonal matrix entries vanish because they vanish on $T$. Hence $\rho$ is an upper triangular representation.
		
		The composition with the projection from the upper triangular matrices to the diagonal is therefore a homomorphism. But since $\O^\times(\G_a)=K^\times$, we have
		\[ \Hom(\G_a,\G_m)=1.\]
		This shows that $\rho$ is indeed unipotent, and the functor (b)$\to$ (c) is well-defined.

		A quasi-inverse (d) $\to$ (c) is given as follows: For every finite-dimensional $K$-vector space $W$, every unipotent automorphism $u$ of $W$ and every $\gamma \in \Z_p$ let
		\begin{align*}
			u^\gamma :=\exp(\gamma\log(u))=\sum_{k=0}^\infty \binom{\gamma}k (u - 1)^k.
		\end{align*}
		Here the second equality again follows from formal properties of $\exp$ and $\log$.
		It is clear from this definition that $u^\gamma$ is continuous in $\gamma$. Moreover, we have $u^{\gamma_1 + \gamma_2} = u^{\gamma_1} u^{\gamma_2}$ by Lemma~\ref{rslt:matrix-log-exp}.(ii). We can thus define the functor (d) $\to$ (c) by $(W, (u_i)_i) \mapsto \rho'$, where $\rho'$ is the representation of $T$ on $W$ with $\rho'(\sum_i \gamma_i e_i) = \prod_i u_i^{\gamma_i}$.
		
		We now construct (c) $\to$ (a), so let $\rho'\colon T \to \GL(W)$ be an object of (c). By \cref{rslt:matrix-log-exp}, the composed map
		\begin{align*}
			\theta := \log \comp \rho'\colon T \to \End(W)    
		\end{align*}
		is continuous, additive and has image in pairwise commuting nilpotent matrices. By continuity, $\theta$ is $\Z_p$-linear. Now by Lemma~\ref{rslt:Hom(G_a)} there is a natural isomorphism
		\[ \Hom_{\cts}(\Z_p,K)=\Hom(\G_a,\G_a).\]
		As $\End(W)$ is a vector space, with associated vector group $\underline{\End}(W)$, it follows that
		\begin{align*}
			\Hom_{\Z_p}(T, \End(W)) = \End(W)^{\rk T} = \Hom_{\mathrm{alg}}(T \tensor_{\Z_p} \mathbb G_a, \underline{\End}(W)).
		\end{align*}
		In particular, $\theta$ extends uniquely to an algebraic map $\theta\colon T \tensor_{\Z_p} \mathbb G_a \to \underline{\End}(W)$. Again by Zariski-density, for every $K$-algebra $R$ the image of $\theta(R)\colon T \tensor_{\Z_p} R \to \underline{\End}(W)(R)$ lies in pairwise commuting nilpotent matrices, so by \cref{rslt:matrix-log-exp} we can compose $\theta$ with $\exp$ (note that $\log$ and $\exp$ are in fact \textit{algebraic} on the unipotent and nilpotent matrices, respectively) to get an algebraic map
		\begin{align*}
			\rho := \exp \comp \theta\colon T \tensor_{\Z_p} \mathbb G_a \to \underline{\Aut}(W),
		\end{align*}
		i.e.\ an object of (a).

		It remains to check that starting with a rigid analytic representation $\rho:T\otimes_{\Z_p}\G_a\to \GL_n$,  sending $(b)\to (c)\to (a)\to (b)$ produces a representation $\rho'$ which is isomorphic to $\rho$. To see this, we note that we can regard $\rho$ and $\rho'$ as rigid analytic functions
		\[\rho,\rho':\G_a\otimes_{\Z_p} T\to \GL(W).\]
		By construction, these agree on $T$.  By the above Zariski-density argument, this implies that they agree on $\G_a\otimes_{\Z_p} T$.
	\end{proof}
	
	Similarly as for unipotent vector bundles, it is not a priori clear, but true, that an analytic vector extension is the same as a $v$-vector extension that is locally free in the analytic topology. This is guaranteed by the following lemma:
	\begin{Lemma} \label{rslt:cartesian-diagram-for-sub-vector-group-of-unipotent}
		Let $A$ be an abeloid variety over $K$, let $U$ be a unipotent algebraic group over $K$, considered as a rigid analytic group via analytification. Let $V \subset U$ be a rigid subgroup that is a vector group. Then the following square is Cartesian and all its maps are injective:
		\begin{center}\begin{tikzcd}
				H^1_\an(A, V) \arrow[r] \arrow[d] & H^1_v(A, V) \arrow[d]\\
				H^1_\an(A, U) \arrow[r] & H^1_v(A, U).
		\end{tikzcd}\end{center}
	\end{Lemma}
	\begin{proof}
		Clearly the two horizontal maps are injective. By \cref{rslt:H1-as-Homs}, the right vertical map can be identified with the map \[\Hom(T_pA, V(K)) \to \Hom(T_pA, U(K))\]
		and is thus also injective.  It follows that all maps in the diagram are injective.
		
		Let $f$ be any element in $H^1_v(A, V)$ whose image in $H^1_v(A, U)$ comes from $H^1_\an(A, U)$. Then there is a cover $\mathfrak W$ of $A$ by affinoid opens $W \subseteq A$ such that the image of $f$ in each $H^1_v(W, U)$ is trivial. Consider the diagram
		\begin{center}\begin{tikzcd}
				H^1_\an(A, V) \arrow[d] \arrow[r] & H^1_v(A, V) \arrow[d,hook]\arrow[r] & \prod_{W \in \mathfrak W} H^1_v(W, V)\arrow[d]\\
				H^1_\an(A, U) \arrow[r] & H^1_v(A, U) \arrow[r] & \prod_{W \in \mathfrak W} H^1_v(W, U).
		\end{tikzcd}\end{center}
		The rows are not necessarily exact, but every element that goes to $0$ on the right comes from the left. We conclude that it suffices to prove that the vertical morphism on the right has trivial fibre over $0$. 
		
		We argue by induction on the dimension of $U$. The induction start of $U=\G_a$ is clear, here we have to have $V=\G_a$ and any non-trivial homomorphism $\G_a\to \G_a$ is an isomorphism by~\cref{rslt:Hom(G_a)}, so the vertical map is an isomorphism.
		
		If $U$ has dimension $>1$, then we can write it as an extension $0 \to U_0 \to U \to U' \to 0$ with $U_0 \cong \mathbb G_a$. We form the pullback:
		\begin{center}\begin{tikzcd}
				0 \arrow[r] & V_0 \arrow[r] \arrow[d] & V \arrow[r] \arrow[d] & V' \arrow[d] \arrow[r] & 0\\
				0 \arrow[r] & U_0 \arrow[r] & U \arrow[r] & U' \arrow[r] & 0
		\end{tikzcd}\end{center}
		where $V'$ is defined as the quotient $V/V_0$, or equivalently the image of $V$ in $U'$. This is again a vector group by \cref{rslt:Hom(G_a)}.
		Then on each $W \in \mathfrak W$ we obtain a diagram of long exact sequences (the top row of abelian groups, the bottom row of pointed sets)
		\[\begin{tikzcd}
			&& H^1_v(W, V_0) \arrow[r] \arrow[d] & H^1_v(W, V) \arrow[d,"\alpha"] \arrow[r] & H^1_v(W, V') \arrow[d]\\
			U(W) \arrow[r] & U'(W) \arrow[r] & H^1_v(W, U_0) \arrow[r, hook,"\beta"] & H^1_v(W, U) \arrow[r] & H^1_v(W, U').
		\end{tikzcd}\]
		We wish to see that $\alpha$ has trivial fibre over $0$. For this it suffices to see that $\beta$ is injective, since the outer vertical maps have trivial fibre over $0$ by the induction hypothesis. But $H^1_\an(W, U_0) = 0$ since $W$ is affinoid, which implies that $U(W) \to U'(W)$ is surjective.
	\end{proof}
	\begin{proof}[Proof of \cref{rslt:characterize-tau-unipotent-VB-via-algebraic-reps}]
		
		We first note that the given functor is fully faithful: It suffices to prove this for $\tau=v$. By \cref{rslt:algebraic-reps-of-vector-groups-equiv-unip-reps-of-Z-p-mod}, the vector bundle $E_\tau A\times^{\rho}\GL_n$ can then equivalently be described as the pushout of $T_pA\to\wt A\to A$ along $\rho_{|T_pA}$. Regarding unipotent representations of $T_pA$ as a full subcategory of all representations, and similarly for $v$-vector bundles, the fully faithfulness then follows from the equivalence in \cref{rslt:profet-vb-vs-Reps}.
		
		Next, to see essential surjectivity,
		let $F$ be a unipotent $\tau$-vector bundle on $A$. We need to see that there is a vector group $V$ such that $F$ is the pushout of some $\tau$-vector extension $0 \to V \to E \to A \to 0$ along some representation $V \to \GL_n$, which is automatically unipotent by  \cref{rslt:algebraic-reps-of-vector-groups-equiv-unip-reps-of-Z-p-mod}. When this is known, the result follows from the universal property of the universal $\tau$-vector extension, \cref{rslt:universal-vector-extension-exists}.
		
		By \cref{rslt:unipotent-trivial-on-wtA}, the pullback of $F$ to $\wt A$ induces a continuous representation $\rho\colon T_p A \to \GL_n(K)$. By \cref{rslt:algebraic-reps-of-vector-groups-equiv-unip-reps-of-Z-p-mod}, this extends uniquely to an algebraic unipotent representation
		\[T_pA \tensor_{\Z_p} \mathbb G_a \to \GL_n\]
		that factors through a vector-subgroup $V$ of $\GL_n$. 
		
		Since $\rho$ factors over $V$, the commutative diagram of Cartan--Leray maps
		\begin{center}\begin{tikzcd}
				\Hom(T_pA,V(K)) \arrow[r,hook] \arrow[d,hook] & H^1_v(A,V) \arrow[d,hook]\\
				\Hom(T_pA,\GL_n(K))\arrow[r,hook] & H^1_v(A,\GL_n)
		\end{tikzcd}\end{center}
		means that $F$ can naturally be seen as a $V$-torsor. But by \cref{rslt:Ext-A-Ga-is-H1-of-O} we have $H^1_v(A, V) = \Ext^1_v(A, V)$, hence $F$ can be endowed with a unique structure of a $v$-vector extension 
		\[0 \to V \to F \to A \to 0.\]
		It is then clear from this construction that $F$ can be recovered from this as the pushout of $E$ along $V \injto \GL_n$, proving the claim in the case $\tau = v$.
		
		Let now $F$ be an analytic vector bundle. We claim that then our extension $E$ from above is in fact an analytic vector extension. To see this, let $U_n$ denote the group of upper triangular unipotent $n \times n$ matrices. Using $\Ext^1_\tau(A, V) = H^1_\tau(A, V)$ for $\tau \in \{ \an, v \}$ by \cref{rslt:Ext-A-Ga-is-H1-of-O}, \cref{rslt:cartesian-diagram-for-sub-vector-group-of-unipotent} implies that we have a Cartesian diagram
		\begin{center}\begin{tikzcd}
				\Ext^1_\an(A, V) \arrow[r,hook] \arrow[d,hook] & \Ext^1_v(A, V) \arrow[d,hook]\\
				H^1_\an(A, U_n) \arrow[r,hook] & H^1_v(A, U_n)
		\end{tikzcd}\end{center}
		Moreover, by \cref{rslt:unipotent-plus-analytic-implies-analytically-unipotent} we have a Cartesian square
		\begin{center}\begin{tikzcd}
				H^1_\an(A, U_n)\arrow[d,hook] \arrow[r,hook] & H^1_v(A, U_n)\arrow[d,hook]\\
				H^1_\an(A, \GL_n) \arrow[r,hook] & H^1_v(A, \GL_n).
		\end{tikzcd}\end{center}
		Combining these shows that the extension $E$ associated to our unipotent vector bundle $F$ is an analytic vector extension if and only if $F$ is analytic.
	\end{proof}
	
	We can summarise the results of this section by saying that we obtain a commutative diagram of functors in which the vertical arrows are fully faithful embeddings and all other arrows are equivalences of categories
	
	
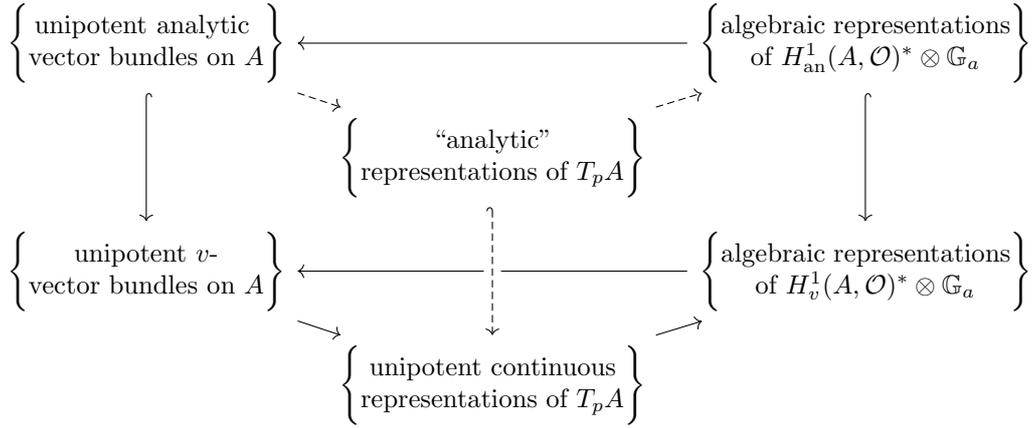
\begin{center}
		\begin{tikzcd}[column sep = 0.2cm,row sep=0.2cm]
			{\Bigg\{\begin{array}{@{}c@{}l}\text{unipotent analytic }\\\text{vector bundles on $A$} \end{array}\Bigg\}} \arrow[rrd,dashed]\arrow[dd, hook] &  &  &  & {\Bigg\{\begin{array}{@{}c@{}l}\text{algebraic representations}\\\text{of $H^1_{\an}(A, \O)^* \tensor \mathbb G_a$} \end{array}\Bigg\}} \arrow[dd, hook] \arrow[llll] \\
			&  &{\Bigg\{\begin{array}{@{}c@{}l}\text{``analytic''}\\\text{representations of $T_pA$}\end{array}\Bigg\}}\arrow[rru,dashed]&  &  \\
			{\Bigg\{\begin{array}{@{}c@{}l}\text{unipotent $v$-}\\\text{vector bundles on $A$} \end{array}\Bigg\}} \arrow[rrd] \arrow[rrrr,leftarrow]&  & &  & {\Bigg\{\begin{array}{@{}c@{}l}\text{algebraic representations}\\\text{of $H^1_v(A, \O)^* \tensor \mathbb G_a$} \end{array}\Bigg\}} \\
			&&{\Bigg\{\begin{array}{@{}c@{}l}\text{unipotent continuous}\\\text{representations of $T_pA$} \end{array}\Bigg\}}\arrow[rru]\arrow[uu,crossing over,hookleftarrow,dashed] &&
		\end{tikzcd}
	\end{center}

	This answers \cref{q:when-is-rep-an} above in the unipotent case (the dashed arrows): In order to detect whether a unipotent representation $\rho:T_pA\to \GL_n(K)$ corresponds to an analytic vector bundle, we extend $K$-linearly to a representation \[\rho\otimes K:T_pA\otimes K\to \GL_n(K),\] and check whether $\rho\otimes K$ factors through  $H^1_{\an}(A, \O)^*$. Equivalently, via the ``dual Hodge--Tate sequence''
	\[0\to H^0(A,\wt\Omega^1)^\ast\to H^1_{v}(A, \O)^*\to H^1_{\an}(A, \O)^*\to 0,\]
	this means that $\rho \otimes K$ vanishes on $H^0(A,\wt\Omega^1)^\ast$.
	\subsection{Universal $\Z_p$-extensions}
	Before going on, we record two alternative points of view that we hope shed some light on the constructions of the last section.
	We begin by giving a different perspective on the $v$-topological part of the last section: 
	
	We have just seen that any representation of a vector group arises as the unique $K$-linear extension of a representation of some $\Z_p$-lattice. There is an analogous statement for $\Z_p$-extensions themselves: 
	\begin{Definition}
		By a $\Z_p$-extension of $A$ we shall mean a short exact sequence of $v$-sheaves
		\[ 0\to T\to E\to A\to 0\]
		where $T\cong \underline{\Z_p}^r$ for some $r\in \N$.
	\end{Definition}
	We have the following analog of the universal vector extension:
	\begin{Proposition}
		There is a universal $\Z_p$-vector extension over $A$ given by the sequence
		\begin{equation}\label{ses:TpA-wtA-A}
			0\to T_pA\to \wt A^{p}\to A\to 0,
		\end{equation}
		i.e.\ every other $\Z_p$-vector extension arises from this via pushout.
	\end{Proposition}
	\begin{proof}
		Let $V$ be finite free over $\underline{\Z}_p$, then we consider the long exact sequence of $\Hom(-,V)$ applied to the displayed sequence: Since $V$ is derived $p$-complete and $\wt A^p$ is $p$-divisible, we have by \cite[Lemma~A.10]{heuer-isoclasses} \[\Ext_v(\wt A^p,V)=0.\]
		Consequently, the boundary map induces an isomorphism
		\[ \Hom_v(T_pA,V)=\Ext^1_v(\wt A^p,V),\]
		which implies the desired result.
	\end{proof}
	\begin{Corollary}
		The universal $v$-vector extension can explicitly be described as the pushout of \cref{ses:TpA-wtA-A} along the natural morphism of $v$-sheaves
		\[ T_pA\to T_pA\otimes_{\Z_p} \G_a=H^1_v(A,\O)^{\ast}\otimes_K\G_a.\]
	\end{Corollary}
	\begin{proof}
		This follows from comparing long exact sequences along the natural transformation $\Hom(-,\Z_p)\to\Hom(-,\O)$.
	\end{proof}
	
	\begin{Remark}
		The analog of this story for the analytic topology is more subtle: One can show using {\cite[Proposition 8.5]{BL-Degenerating-AV}} that
		\[ \Ext^1_{\an}(A,\Z_p)=H^1_{\an}(A,\Z_p)=H^1_{\an}(A,\Z)\otimes \Z_p=\Hom(\G_m,A^\vee)\otimes \Z_p.\]
		The last group depends on the reduction type of $A$, namely its rank is precisely the torus rank of the semi-stable reduction. In terms of the Raynaud uniformisation $A=E/M$, the universal analytic $\Z_p$-extension is then given by the canonical pro-\'etale tower
		\[\varprojlim_n E/p^nM\to E/M.\]
		If $E$ is totally degenerate, e.g.\ if $A$ is a Tate curve, this is an example of a canonical tower for ordinary $A$ that we discuss in the next section.
	\end{Remark}
	\subsection{The case of ordinary reduction}
	Finally, we now briefly discuss the special case that $A$ has ordinary reduction, since the discussion of this section simplifies a great deal in this case. In particular, we explain an alternative proof of Theorem~\ref{rslt:characterize-tau-unipotent-VB-via-algebraic-reps} in the ordinary reduction case. Note that by \cite{Luetkebohmert_abeloid} the usual notion of ordinary reduction generalises to the abeloid case:
	\begin{Definition}
		We say that the abeloid variety $A$ has ordinary reduction if the abelian part of the semi-stable reduction is an ordinary abelian variety over the residue field.
	\end{Definition}
	
	Note that the abeloid variety $A$ has ordinary reduction if and only if its Hodge filtration is already rational: Namely, in this case there exists a $p$-divisible subgroup $C\subseteq A[p^\infty]$ of height $d=\dim A$ called the canonical subgroup such that the Tate module $T_pC\subseteq T_pA$ spans the Hodge filtration, i.e.\ we have a commutative diagram of the form 
	\[ \begin{tikzcd}
		T_pC\otimes_{\Z_p}K \arrow[d, "\sim"labelrotate] \arrow[r, hook] & T_pA\otimes_{\Z_p}K \arrow[d, "\sim"labelrotate] \\
		{H^0(A,\wtOm^1)^\ast} \arrow[r, hook] & {H^1_v(A,\O)^\ast.}
	\end{tikzcd}\]
	In light of this diagram, \cref{rslt:characterize-tau-unipotent-VB-via-algebraic-reps} implies the following simpler statement:
	
	\begin{Corollary}
		Let $A$ be an abeloid variety with ordinary reduction, let $\rho:T_pA\to \GL_n(K)$ be a continuous representation and let $V_{\rho}$ be the associated $v$-vector bundle from \cref{rslt:profet-vb-vs-Reps}.
		\begin{enumerate}
			\item Suppose $\rho$ is unipotent. Then $V_\rho$ is analytic if and only if $\rho$ vanishes on $T_pC\subseteq T_pA$.
			\item In general, $V_\rho$ is analytic if and only if $\rho({T_pC})$ is finite.
		\end{enumerate}
	\end{Corollary}
	\begin{proof}
		The case of unipotent bundles follows from  \cref{rslt:characterize-tau-unipotent-VB-via-algebraic-reps} via the above diagram. To deduce part 2, it remains by
		\cref{rslt:decomposition-of-profet-vb-category} to treat the case of line bundles. We may thus assume that $\rho:T_pA\to K^\times$ is a character. In this case, the Corollary follows from the results of \cite{heuer-v_lb_rigid}: As recalled in \cref{sec:prelim-simpson-for-line-bundles}, the $v$-line bundle $V_{\rho}$ is analytic if and only if $\HT\log(V_{\rho})=0$. Combining the diagram \cite[\S4.4, (11)]{heuer-v_lb_rigid} with the above diagram, we see that $\HT\log(V_{\rho})$ gets identified with the image of $\rho$ under the natural map
		\[ \Hom(T_pA,K^\times)\xrightarrow{\log}\Hom(T_pA,K)\xrightarrow{\mathrm{res}} \Hom(T_pC,K)\xisomarrow{\HT}H^0(A,\wtOm^1).\]
		This vanishes if and only if $\rho(T_pC)$ lands in $\mu_{p^\infty}\subseteq K^\times$, the kernel of the logarithm.
	\end{proof}
	We now sketch how one can give an alternative proof of the unipotent part of the Corollary by interpreting it in more geometric terms:
	Consider the intermediate pro-\'etale cover
	\[ \wt A^c:=\wt A^{p}/T_pC\to A.\]
	This can be identified with the inverse limit over the ``dual-to-canonical'' isogenies
	\[\dots A/C[p^2]\to A/C[p]\to A.\]
	As usual, we see by the Cartan--Leray sequence that we have 
	$H^1_v(\wt A^c,\O)=\Hom(T_pC,K)$
	and we thus obtain a short exact sequence
	\[0\to H^1_{\an}(A,\O)\to H_v^1(A,\O)\to  H_v^1(\wt A^c,\O)\to 0.\]
	We claim that in fact, the following stronger statement holds:
	\begin{equation}\label{eq:H^1_an(wt A^c}
		H_{\an}^1(\wt A^c,\O)=0.
	\end{equation}
	We sketch a proof: As in Lemma~\ref{rslt:Ext-A-Ga-is-H1-of-O} it follows from Lemma~\ref{rslt:BD-in-deg-1} that one can reduce to showing $\Ext^1_{\an}(\wt A^c,\O)=0$.
	Via the Raynaud uniformisation $A=E/M$, one uses this to reduce to the case of good reduction: Namely, the dual-to-canonical isogenies lift uniquely to a tower over $E$. Let $\wt E^c$ be its limit. Exactly like in \cite[Theorem 4.6]{perfectoid-covers-Arizona}, one sees that in the limit the morphism between the two towers gives a short exact sequence relating $\wt A^c$ and  $\wt E^c$. This reduces us to showing $\Ext^1_{\an}(\wt E^c,\O)=0$. For this  we use that $\wt E^c$ is itself an extension by a torus of the cover $\wt B^c$ for an abelian variety $B$ of good ordinary reduction. Since $\Ext^1(\G_m,\O)=0$ as one sees from comparing $\G_m$ to any Tate curve, this reduces us to the case of good reduction.
	
	In this case, one can use the smooth formal model $\mathfrak A$ of $A$, which induces a formal model $\wt{\mathfrak A}^c$ of $\wt A^c$. By a comparison of cohomologies, it suffices to prove that $H^1(\wt{\mathfrak A}^c,\O)=0$. For this we use that the dual-to-canonical tower reduces mod $p$ to the tower of Verschiebung isogenies. We thus have
	\[ H^1(\mathfrak A,\O/p)=\varinjlim_{V}H^1((\mathfrak A/p)^{(p^n)},\O/p)=0\]
	since Verschiebung kills each cohomology group.
	
	It follows from \cref{eq:H^1_an(wt A^c} by induction that pullback along $\wt A^c\to A$ kills precisely the unipotent analytic vector bundles. This gives an independent proof that a unipotent $v$-vector bundle is analytic if and only if its associated $T_pA$-representation is trivial on $T_pC$.
	
	\begin{Remark}
		One way in which this perspective could be helpful is that in contrast to the approach via universal vector extensions, it does not use the group structure on $A$ in an essential way. In particular, we believe that diamantine covers like $\wt A^c$ can also help understand when generalised representations are representations beyond the case of abeloids.
	\end{Remark}
	
	%------------------------------------------------------------------------------
	\section{Higgs bundles on abeloids}
	
	As before, let $K$ be a complete algebraically closed field extension of $\Q_p$. Having studied pro-finite-\'etale $v$-vector bundles on an abeloid variety $A$ over $K$, we now turn our focus to the other side of the Corlette--Simpson correspondence: the Higgs bundles. It turns out that pro-finite-\'etale Higgs bundles admit a similar decomposition as pro-finite-\'etale $v$-vector bundles, as predicted by the $p$-adic Corlette--Simpson philosophy. Proving this is the main goal of this subsection.
	
	\begin{Proposition} \label{rslt:Higgs-on-A-are-unipotent-times-lb}
		Let $A$ be an abeloid variety over $K$. Then a Higgs bundle $(E, \theta)$ on $A$ is pro-finite-\'etale if and only if it can be written as
		\begin{align*}
			(E, \theta) = \bigoplus_{i=1}^n (L_i, \theta_{L_i}) \tensor (U_i, \theta_{U_i}),
		\end{align*}
		where each $(L_i, \theta_{L_i})$ is a pro-finite-\'etale Higgs line bundle and each $(U_i, \theta_{U_i})$ is a unipotent Higgs bundle on $X$.
	\end{Proposition}
	
	As in the proof of the analogous result \cref{rslt:vb-on-A-are-unipotent-times-lb}, \cref{rslt:Higgs-on-A-are-unipotent-times-lb} relies on the following characterization of extensions of Higgs line bundles:
	
	\begin{Lemma} \label{rslt:no-Ext-of-Higgs-line-bundles}
		Let $(L, \theta) \not\cong (L', \theta')$ be two pro-finite-\'etale Higgs line bundles on an abeloid $A$ over $K$. Then
		\begin{align*}
			\Hom((L, \theta), (L', \theta')) &=0,\\
			\Ext^1((L, \theta), (L', \theta')) &=0.
		\end{align*}
		Here $\Ext^1$ denotes the set of isomorphism classes of extensions of Higgs bundles on $A$.
	\end{Lemma}
	\begin{proof}
		By  Lemma~\ref{rslt:H^1_v(A-L)}, if $L \ne L'$ then $\Hom(L,L')=0$ and  there are then no non-trivial extensions of $L$ by $L'$. In particular, any Higgs field on $L\oplus L'$ decomposes. Thus also $\Ext^1=0$. We can thus reduce to the case $L = L'$. Tensoring with $(L'^{-1}, -\theta')$ we are reduced to the case $L = L' = \O$ and $\theta' = 0$. The first statement then follows since $\Hom(\O,\O)=H^0(A,\O)=K$, and in particular any non-trivial morphism $\O\to \O$ is an isomorphism, which implies $\theta=\theta'$. 
		
		To prove the second statement, we consider an extension
		\begin{align*}
			0 \to (\O, 0) \to (E, \Theta) \to (\O, \theta) \to 0.
		\end{align*}
		Fix an isomorphism $\wt\Omega^1_A \cong \O_A^g$. Then
		\begin{align*}
			H^0(A, \IEnd(E) \tensor \wt\Omega^1) \cong H^0(A, \IEnd(E))^g = \End(E)^g,
		\end{align*}
		so the Higgs field $\Theta$ can be viewed as a collection of $g$ endomorphisms $\Theta_1, \dots, \Theta_g$ of $E$. The condition $\Theta \wedge \Theta = 0$ translates to the condition that the $\Theta_i$ commute pairwise. 
		
		Fix now moreover a splitting of $E$ over $\wt A$, so that $E(\wt A)=K^2$. Then via \cref{rslt:profet-vb-vs-Reps}, the vector bundle $E$ is associated to a representation \[\rho\colon TA = \pi_1(A, 0) \to U_2(K)=\left\{\smallmat{1}{\ast}{0}{1}\right\}\cong K,\] where $U_2(K)$ denotes the group of upper triangular unipotent $2 \times 2$-matrices over $K$. Then $\End(E)$ corresponds precisely to the endomorphisms of $K^2$ (i.e.\ $2 \times 2$-matrices) which commute with all $\rho(x)$ for $x \in TA$. Similarly, the Higgs field $\theta$ on $\O$ corresponds to $g$ elements $\theta_1, \dots, \theta_g \in K$. The fact that $(E, \Theta)$ is an extension of $(\O, \theta)$ by $(\O, 0)$ forces
		\begin{align*}
			\Theta_i = \begin{pmatrix} 0 & b_i\\ 0 & \theta_i \end{pmatrix}
		\end{align*}
		for $i = 1, \dots, g$ and some $b_i \in K$. The condition that these commute is then equivalent to
		\[b_i\theta_j=b_j\theta_i\quad\text{ for all }1\leq i,j\leq n.\]
		Assume that  $\theta\ne 0$, i.e.\ after reordering that $\theta_1\neq 0$. By conjugating everything by $\left(\begin{smallmatrix} 1 & b_1/\theta_1\\ 0 & 1 \end{smallmatrix}\right)$, which is an automorphism of  $(E, \Theta)$ as an extension, we can arrange that $b_1 = 0$. But then setting $i=1$ in the above commutativity condition implies $b_j=0$ for all $j$. 
		
		Similarly, a matrix $\smallmat{1}{b}{0}{1}$ commutes with $\smallmat{0}{0}{0}{\theta_1}$ if and only if $b=0$, so the condition that the representation $\rho\colon TA \to U_2(K)$ commutes with $\Theta_1$ forces $\rho$ to be trivial. Thus indeed, $(E, \Theta)$ is a trivial extension.
	\end{proof}
	
	\begin{proof}[Proof of \cref{rslt:Higgs-on-A-are-unipotent-times-lb}]
		Only the ``only if'' direction needs proof, so assume that $E$ is pro-finite-\'etale. As in the proof of \cref{rslt:no-Ext-of-Higgs-line-bundles} we can fix an isomorphism $\wt\Omega^1_A \cong \O_A^g$ and a basis $E(\wt A)=K^n$ which allows us to view $(E, \theta)$ equivalently as a representation $\rho\colon TA \to \GL_n(K)$ plus $g$ matrices $\theta_1, \dots, \theta_g\in M_n(K)$ such that all $\theta_i$ commute with all matrices in the image of $\rho$ and the $\theta_i$'s commute pairwise. Thus there is a simultaneous eigenvector $w \in K^n$ for the collection of all $\rho(x)$ and all $\theta_i$. Being an eigenvector for the $\rho(x)$ means that $w$ corresponds to a pro-finite-\'etale $v$-line bundle $L \subset E$, which by \cref{rslt:analyticity-of-sub-LB-of-VB} is automatically analytic. Being an eigenvector of the endomorphisms $\theta_1, \dots, \theta_g$ of $K^n$ means that $\theta$ preserves $L$ and hence defines a Higgs field $\theta_L$ on $L$ such that $(L, \theta_L) \injto (E, \theta)$ is a morphism of Higgs bundles. The quotient $E/L$ is again a pro-finite-\'etale vector bundle by \cref{rslt:profet-vb-vs-Reps}, and it is analytic by \cref{rslt:v-loc-free-implies-analytic-loc-free}. Inductively, we deduce that $(E, \theta)$ is a successive extension of pro-finite-\'etale Higgs line bundles. Thus the claim follows from \cref{rslt:no-Ext-of-Higgs-line-bundles}.
	\end{proof}
	
	We can reformulate \cref{rslt:Higgs-on-A-are-unipotent-times-lb} in the following form, analogously to \cref{rslt:decomposition-of-profet-vb-category}:
	
	\begin{Corollary} \label{rslt:decomposition-of-profet-Higgs-category}
		Let $A$ be an abeloid variety over $K$ and let $\Higgs^{\pf}_1(A)$ denote the set of isomorphism classes of pro-finite-\'etale Higgs line bundles on $A$. Then there is a natural equivalence of categories
		\begin{align*}
			\{ \text{pro-finite-\'etale Higgs bundles on $A$} \} \ &\cong \ \bigoplus_{L \in \Higgs^\pf_1(A)} \{ \text{unipotent Higgs bundles on A} \},\\
			\bigoplus_{[L]} L \tensor U_{[L]} \ & \mapsfrom \ (U_{[L]})_{[L]}.
		\end{align*}
		Here we can associate to every isomorphism class in $\Higgs^{\pf}_1(A)$ a canonical representative $L$ which is determined by saying that its underlying line bundle is that from \cref{rslt:decomposition-of-profet-vb-category}.
	\end{Corollary}
	\begin{proof}
		By \cref{rslt:Higgs-on-A-are-unipotent-times-lb} the functor from right to left is essentially surjective. To get fully faithfulness, by the same arguments as in the proof of \cref{rslt:decomposition-of-profet-vb-category} it is enough to note that by \cref{rslt:no-Ext-of-Higgs-line-bundles}, for any two pro-finite-\'etale Higgs line bundles $(L_1, \theta_1) \not\cong (L_2, \theta_2)$ we have $\Hom((L_1, \theta_1), (L_2, \theta_2)) = 0$. 
	\end{proof}
	
	\begin{Remark}\label{remark:nfr-vs-profet}
		If $A$ is an abeloid of good reduction, then by \cite[Theorem~6.1]{LangerFundamental} there is an analogous description of the category of analytic vector bundles with numerically flat reduction on $A$, and thus of Higgs bundles with numerically flat reduction: Namely, the condition that $L$ is topological torsion gets replaced by the condition that $L\in \Pic^0(A)$.
		
		Over the base field $K=\C_p$, we have $\Pic^{\tt}(A)=\Pic^0(A)$ by \cite[Lemma~B.5]{heuer-diamantine-Picard}. Therefore, in this case, the pro-finite-\'etale Higgs bundles are precisely the ones with numerically flat reduction, in line with \cite[Theorem 1.2]{wuerthen_vb_on_rigid_var}. However, over more general base fields, the two notions are different in the case of line bundles: In general, one has $\Pic^\tt\subseteq \Pic^0$, so being pro-finite-\'etale is stronger than having numerically flat reduction.
	\end{Remark}
	%------------------------------------------------------------------------------
	\section{The $p$-adic Corlette--Simpson correspondence}
	
	We now have everything in place to prove our main result,  the $p$-adic Corlette--Simpson correspondence for abeloid varieties. As before, let $K$ be a complete algebraically closed field extension of $\Q_p$. Recall that for any smooth proper rigid space $X$ over $K$, its associated Hodge--Tate spectral sequence degenerates \cite[Theorem~13.3]{BMS}, and hence induces a short exact sequence (cf. \cref{sec:prelim-higgs-bundles})
	\begin{align*}
		0 \to H^1_\an(X, \O) \to H^1_v(X, \O) \xto{\HT} H^0(X, \wt\Omega^1) \to 0.
	\end{align*}
	The occurring map $\HT$ is called the \emph{Hodge--Tate map} of $X$.
	
	\begin{Theorem} \label{rslt:Simpson-correspondence}
		Let $K$ be a complete algebraically closed field extension of $\mathbb{Q}_p$ and let $A$ be an abeloid variety over $K$. Then any choice of an exponential on $K$ and a splitting of the Hodge-Tate map $\HT$ of $A$ induce an exact tensor equivalence of categories
		\begin{align*}
			\mathrm{Rep}_{K}(\pi_1(A,0)) \ \cong \ \{ \text{pro-finite-\'etale Higgs bundles on $A$} \}.
		\end{align*}
	\end{Theorem}	By \cref{rslt:profet-vb-vs-Reps}, the left hand side is equivalent to pro-finite-\'etale $v$-vector bundles.
	As reviewed in \cref{rslt:Simpson-linebundles}, the case of line bundles is known. By our decomposition results of both sides, \cref{rslt:decomposition-of-profet-vb-category} and \cref{rslt:decomposition-of-profet-Higgs-category}, we are left to treat the unipotent case:
	
	\begin{Proposition} \label{rslt:Simpson-correspondence-for-unipotent}
		The choice of a splitting of the Hodge-Tate map $\HT$ of $A$ induces an equivalence of categories
		\begin{align*}
			\{ \text{unipotent $v$-bundles on $A$} \} \ \cong \ \{ \text{unipotent Higgs bundles on $A$} \}.
		\end{align*}
	\end{Proposition}
	\begin{proof}
		By \cref{rslt:characterize-tau-unipotent-VB-via-algebraic-reps}, unipotent $v$-bundles on $A$ are equivalent to algebraic representations
		\[ H^1_v(A, \O)^*\to \underline{\Aut}(W)\]
		on finite dimensional $K$-vector spaces. Here and in the following, we simply write $H^1_v(A, \O)^*$ for the algebraic group $H^1_v(A, \O)^*\otimes \G_a$ to ease notation. We proceed similarly for the other terms in the Hodge-Tate sequence. This abuse of notation is harmless due to the category equivalence in  \cref{rslt:algebraic-reps-of-vector-groups-equiv-unip-reps-of-Z-p-mod}.

		We now extend this to a chain of equivalences of categories as follows:
		\[\begin{tikzcd}[ampersand replacement=\&,row sep = 0.55cm]
			{\left\{ \text{unipotent $v$-vector bundles on $A$}\right\}} \arrow[d,"(a)"] \\
			{\left\{(W,\rho)\middle| \begin{array}{@{}l@{}l}W &:\text{finite dimensional $K$-vector space,}\\
					\rho &:\text{$H^1_v(A, \O)^*\to \underline{\Aut}(W)$ algebraic representation}
				\end{array}\right\}}  \arrow[d, dashed,"(b)"] \\
			{\left\{(W,\rho_1,\rho_2)\middle| \begin{array}{@{}l@{}l @{\:}l}W &:&\text{finite dimensional $K$-vector space,}\\
					\rho_1 &:&\text{$H^1_\an(A, \O)^*\to \underline{\Aut}(W)$ algebraic representation,}\\
					\rho_2 &:& \text{$H^0(A, \wtOm^1)^* \to  \underline{\Aut}(W)$ algebraic representation}\\
					&&\text{ such that $\rho_1$ and $\rho_2$ commute}
				\end{array}\right\}}   \arrow[d,"(c)"] \\
			{\left\{(W,\rho_1,\theta)\middle| \begin{array}{@{}l@{}l @{\:}l}W &:&\text{finite dimensional $K$-vector space,}\\
					\rho_1 &:&\text{$H^1_\an(A, \O)^*\to \underline{\Aut}(W)$ algebraic representation,}\\
					\theta &:& H^0(A, \wtOm^1)^* \to \underline{\End}(W)  \text{ $K$-linear map whose image consists of}\\
					&&\text{pairwise commuting nilpotent matrices that commute with $\rho_1$}
				\end{array}\right\}}\arrow[d,"(d)"] \\
			{\left\{ \text{unipotent Higgs bundles on $A$}\right\}}
		\end{tikzcd}\]
		
		where $(a)$, $(c)$ and $(d)$ are canonical and $(b)$ depends  on the splitting of $\HT$.
		
		The equivalence $(a)$ is \cref{rslt:characterize-tau-unipotent-VB-via-algebraic-reps}.
		
		The equivalence $(b)$ is given by the splitting of $\HT$ on $A$ which induces an isomorphism
		\[ H^1_v(A, \O)^*=H^1_\an(A, \O)^*\oplus H^0_\an(A, \wtOm^1)^*\]
		of $K$-vector spaces. In particular, an algebraic representation of the associated vector group on the left is the same as commuting representations of the two factors on the right.
		
		The equivalence (c) is given by
		\begin{align*}
			(W, \rho_1, \rho_2) \mapsto (W, \rho_1, \theta := \log \comp \rho_2),
		\end{align*}
		where $\log$ denotes the matrix logarithm from \cref{rslt:matrix-log-exp}.  Then $\theta\colon H^0(A, \wt\Omega^1)^* \to \underline{\End}(W)$ is an algebraic map and in particular $K$-linear, so $(V, \rho_1, \theta)$ is indeed an element of the right hand side. The functor thus constructed is an equivalence because it admits a quasi-inverse
		\begin{align*}
			(W, \rho_1, \theta) \mapsto (W, \rho_1, \rho_2 := \exp \comp \theta).
		\end{align*}
		Here the commutativity conditions are equivalent to each other since  by Lemma~\ref{rslt:matrix-log-exp}.(iii), two unipotent matrices $u$ and $u'$ commute if and only if $u$ and $n:=\log(u')$ commute.
		
		It remains to construct the equivalence $(d)$. Starting with a triple $(W, \rho_1, \theta)$, we see that the pair $(W, \rho_1)$ corresponds to a unipotent analytic vector bundle $E$ on $A$ by \cref{rslt:characterize-tau-unipotent-VB-via-algebraic-reps}. The condition that $\theta$ and $\rho_1$ commute means precisely that $\theta$ can equivalently be seen as a $K$-linear map $\theta\colon H^0(A, \wtOm^1)^* \to \End(E)$. As  $\wtOm^1$ is free, we have
		\begin{align*}
			\theta \in \Hom(H^0(A, \wtOm^1)^*, \End(E)) = H^0(A, \IEnd(E) \tensor \wtOm^1).
		\end{align*}
		The condition that the image of $\theta$ consists of pairwise commuting matrices is equivalent to $\theta \wedge \theta = 0$. Thus $(E, \theta)$ is a Higgs bundle on $A$. Finally, $\theta$ consisting of nilpotent matrices is equivalent to $(E, \theta)$ being unipotent, as desired.
	\end{proof}
	
	\begin{proof}[Proof of \cref{rslt:Simpson-correspondence}]
		By \cref{rslt:profet-vb-vs-Reps} we are reduced to finding an equivalence of pro-finite-\'etale $v$-bundles and pro-finite-\'etale Higgs bundles. By \cref{rslt:Simpson-linebundles}, the choice of exponential and of a splitting of $\HT$ induces a bijection between the sets of isomorphism classes of pro-finite-\'etale $v$-line bundles on $A$ and pro-finite-\'etale Higgs line bundles on $A$, respectively. Thus by \cref{rslt:decomposition-of-profet-vb-category} and \cref{rslt:decomposition-of-profet-Higgs-category} we are reduced to finding an equivalence of unipotent $v$-bundles and Higgs bundles on $A$. This is \cref{rslt:Simpson-correspondence-for-unipotent}.
	\end{proof}
	
	\section{Open questions}
	We end this paper with some open questions to which we hope to return in the future.
	\begin{enumerate}
		
		\item We do not currently know if the property of an analytic vector bundle on a smooth proper rigid space to be pro-finite-\'etale can be expressed in more classical terms.  As explained in Remark~\ref{remark:nfr-vs-profet}, the pro-finite-\'etale bundles on an abeloid $A$ of good reduction over $\C_p$  are precisely those with numerically flat reduction. However, over more general base fields, being pro-finite-\'etale is stronger.
		
		\item For abelian varieties over $\C_p$, Remark \ref{rem:homogeneous} says that the pro-finite-\'etale analytic vector bundles are precisely the homogeneous ones. Is this also true for abeloids?  
		\item One natural question is whether our $p$-adic Corlette--Simpson correspondence can be upgraded to an isomorphism of moduli spaces, say of $v$-stacks. This is possible in the case of line bundles \cite{heuer-diamantine-Picard}, and it seems that similar methods should apply here.
		
		\item Can we extend our correspondence to the case of not necessarily pro-finite-\'etale bundles, e.g. to a correspondence between Higgs bundles and $v$-vector bundles?

	\end{enumerate}

	\bibliographystyle{alpha}
	\bibliography{bibliography.bib}

\newcommand{\etalchar}[1]{$^{#1}$}
\begin{thebibliography}{BGH{\etalchar{+}}18}

\bibitem[AGT16]{AGT}
Ahmed Abbes, Michel Gros, and Takeshi Tsuji.
\newblock {\em The {$p$}-adic {S}impson correspondence}, volume 193 of {\em
  Annals of Mathematics Studies}.
\newblock Princeton University Press, Princeton, NJ, 2016.

\bibitem[BBM82]{BBM}
Pierre Berthelot, Lawrence Breen, and William Messing.
\newblock {\em {T}h\'{e}orie de {D}ieudonn\'{e} cristalline {{I}I}}, volume 930
  of {\em Lecture Notes in Mathematics}.
\newblock Springer-Verlag, Berlin, 1982.

\bibitem[BGH{\etalchar{+}}18]{perfectoid-covers-Arizona}
Clifford Blakestad, Damián Gvirtz, Ben Heuer, Daria Shchedrina, Koji Shimizu,
  Peter Wear, and Zijian Yao.
\newblock {P}erfectoid covers of abelian varieties.
\newblock {\em Preprint, arXiv:1804.04455}, 2018.

\bibitem[BGR84]{BGR}
Siegfried Bosch, Ulrich G\"{u}ntzer, and Reinhold Remmert.
\newblock {\em {N}on-{A}rchimedean analysis}, volume 261 of {\em Grundlehren
  der Mathematischen Wissenschaften}.
\newblock Springer-Verlag, Berlin, 1984.

\bibitem[BL91]{BL-Degenerating-AV}
Siegfried Bosch and Werner L\"{u}tkebohmert.
\newblock {D}egenerating abelian varieties.
\newblock {\em Topology}, 30(4):653--698, 1991.

\bibitem[BMS18]{BMS}
Bhargav Bhatt, Matthew Morrow, and Peter Scholze.
\newblock {I}ntegral {$p$}-adic {H}odge theory.
\newblock {\em Publ. Math. Inst. Hautes \'{E}tudes Sci.}, 128:219--397, 2018.

\bibitem[Bri12]{Brion_homogVB_AV}
Michel Brion.
\newblock Homogeneous vector bundles over abelian varieties.
\newblock {\em J. Ramanujan Math. Soc.}, 27:91--118, 2012.

\bibitem[BSU13]{BPU_Chennai}
Michel Brion, Preena Samuel, and V.~Uma.
\newblock {\em Lectures on the structure of algebraic groups and geometric
  applications}, volume~1 of {\em CMI Lecture Series in Mathematics}.
\newblock Hindustan Book Agency, New Delhi; Chennai Mathematical Institute
  (CMI), Chennai, 2013.

\bibitem[CBC{\etalchar{+}}19]{Lectures_Arizona}
Bryden Cais, Bhargav Bhatt, Ana Caraiani, Kiran~S. Kedlaya, Peter Scholze, and
  Jared Weinstein.
\newblock {\em {P}erfectoid {S}paces: {L}ectures from the 2017 {A}rizona
  {W}inter {S}chool}, volume 242.
\newblock American Mathematical Soc., 2019.

\bibitem[dJ{\etalchar{+}}20]{StacksProject}
Aise~Johan de~Jong et~al.
\newblock {T}he stacks project.
\newblock 2020.

\bibitem[DW05a]{DeningerWerner-lb_and_p-adic-characters}
Christopher Deninger and Annette Werner.
\newblock {L}ine bundles and {$p$}-adic characters.
\newblock In {\em Number fields and function fields---two parallel worlds},
  volume 239 of {\em Progr. Math.}, pages 101--131. Birkh\"{a}user Boston,
  Boston, MA, 2005.

\bibitem[DW05b]{DeningerWerner_vb_p-adic_curves}
Christopher Deninger and Annette Werner.
\newblock {V}ector bundles on {$p$}-adic curves and parallel transport.
\newblock {\em Ann. Sci. \'{E}cole Norm. Sup. (4)}, 38(4):553--597, 2005.

\bibitem[DW20]{DeningerWerner_vb_p-adic_varieties}
Christopher Deninger and Annette Werner.
\newblock Parallel transport for {v}ector bundles on {$p$}-adic varieties.
\newblock {\em J. Alg. Geom.}, 29:1--52, 2020.

\bibitem[Fal05]{Faltings_SimpsonI}
Gerd Faltings.
\newblock {A} {$p$}-adic {S}impson correspondence.
\newblock {\em Adv. Math.}, 198(2):847--862, 2005.

\bibitem[Guo19]{guo2019hodgetate}
Haoyang Guo.
\newblock Hodge--{T}ate decomposition for non-smooth spaces.
\newblock {\em Preprint, arXiv:1909.09917}, 2019.

\bibitem[Heu21a]{heuer-diamantine-Picard}
Ben Heuer.
\newblock {D}iamantine {P}icard functors of rigid spaces.
\newblock {\em Preprint, arXiv:2103.16557}, 2021.

\bibitem[Heu21b]{heuer-Picard-good-reduction}
Ben Heuer.
\newblock Line bundles on perfectoid covers: the case good reduction.
\newblock {\em Preprint, arXiv:2105.05230}, 2021.

\bibitem[Heu21c]{heuer-v_lb_rigid}
Ben Heuer.
\newblock {L}ine bundles on rigid spaces in the $v$-topology.
\newblock {\em Preprint, arXiv:2012.07918}, 2021.

\bibitem[Heu21d]{heuer-isoclasses}
Ben Heuer.
\newblock pro-\'etale uniformisation of abelian varieties.
\newblock {\em Preprint, arXiv:2105.12604}, 2021.

\bibitem[HK20]{HK_sheafiness}
David Hansen and Kiran~S. Kedlaya.
\newblock Sheafiness criteria for {H}uber rings, 2020.
\newblock https://kskedlaya.org/papers/criteria.pdf.

\bibitem[HL20]{Hansen_Li}
David Hansen and Shizhang Li.
\newblock Line bundles on rigid varieties and {H}odge symmetry.
\newblock {\em Math. Zeitschrift}, 296:1777--1786, 2020.

\bibitem[KL16]{relative-p-adic-hodge-2}
Kiran~S. {Kedlaya} and Ruochuan {Liu}.
\newblock {Relative $p$-adic {H}odge theory, II: {I}mperfect period rings}.
\newblock {\em Preprint, arXiv:1602.06899}, 2016.

\bibitem[Lan12]{LangerFundamental}
Adrian Langer.
\newblock On the {S}-fundamental group scheme. {II}.
\newblock {\em J. Inst. Math. Jussieu}, 11(4):835--854, 2012.

\bibitem[L{\"u}t95]{Luetkebohmert_abeloid}
Werner L{\"u}tkebohmert.
\newblock On the structure of proper rigid groups.
\newblock {\em J. reine angew. Math.}, 468:167--219, 1995.

\bibitem[LZ17]{LiuZhu_RiemannHilbert}
Ruochuan Liu and Xinwen Zhu.
\newblock Rigidity and a {R}iemann-{H}ilbert correspondence for {$p$}-adic
  local systems.
\newblock {\em Invent. Math.}, 207(1):291--343, 2017.

\bibitem[Miy73]{Miyanishi}
Masayoshi Miyanishi.
\newblock Some remarks on algebraic homogeneous vector bundles.
\newblock In {\em Number theory, algebraic geometry and commutative algebra},
  pages 71--93. Kinokuniya, Tokyo, 1973.

\bibitem[MM60]{Matsushima-Morimoto}
Yoz\^{o} Matsushima and Akihiko Morimoto.
\newblock Sur certains espaces fibr\'{e}s holomorphes sur une vari\'{e}t\'{e}
  de {S}tein.
\newblock {\em Bull. Soc. Math. France}, 88:137--155, 1960.

\bibitem[MN84]{MehtaNori}
Vikram~B. Mehta and Madhav~V. Nori.
\newblock Semistable sheaves on homogeneous spaces and abelian varieties.
\newblock {\em Proc. Indian Acad. Sci. Math. Sci.}, 93(1):1--12, 1984.

\bibitem[Mor59]{Morimoto}
Akihiko Morimoto.
\newblock Sur la classification des espaces fibr\'{e}s vectoriels holomorphes
  sur un tore complexe admettant des connexions holomorphes.
\newblock {\em Nagoya Math. J.}, 15:83--154, 1959.

\bibitem[Muk78]{Mukai_semi-homog-VB_on_AV}
Shigeru Mukai.
\newblock Semi-homogeneous vector bundles on an {A}belian variety.
\newblock {\em J. Math. Kyoto Univ.}, 18(2):239--272, 1978.

\bibitem[Mum74]{Mumford}
David Mumford.
\newblock {\em Abelian Varieties}.
\newblock Tata Insitute of Fundamental Research studies in mathematics, 5.
  Oxford University Press, 1974.

\bibitem[MW20]{MannWerner_LocSys_p-adVB}
Lucas Mann and Annette Werner.
\newblock {L}ocal systems on diamonds and $p$-adic vector bundles.
\newblock {\em Preprint, arXiv:2005.06855}, 2020.

\bibitem[NSW08]{NeuSchWin}
J\"{u}rgen Neukirch, Alexander Schmidt, and Kay Wingberg.
\newblock {\em {C}ohomology of number fields}, volume 323 of {\em Grundlehren
  der Mathematischen Wissenschaften}.
\newblock Springer-Verlag, Berlin, second edition, 2008.

\bibitem[Ros58]{rosenlicht-universal}
Maxwell Rosenlicht.
\newblock {E}xtensions of vector groups by abelian varieties.
\newblock {\em Am. J. of Math.}, 80:685--714, 1958.

\bibitem[Sch13a]{Scholze_p-adicHodgeForRigid}
Peter Scholze.
\newblock {$p$}-adic {H}odge theory for rigid-analytic varieties.
\newblock {\em Forum Math. Pi}, 1:e1, 77, 2013.

\bibitem[Sch13b]{Scholze2012Survey}
Peter Scholze.
\newblock {P}erfectoid spaces: a survey.
\newblock In {\em Current developments in mathematics 2012}, pages 193--227.
  Int. Press, Somerville, MA, 2013.

\bibitem[Sch17]{etale-cohomology-of-diamonds}
Peter Scholze.
\newblock Étale cohomology of diamonds.
\newblock {\em Preprint, arXiv:1709.07343}, 2017.

\bibitem[Sch19]{condensed-mathematics}
Peter Scholze.
\newblock Lectures on {Condensed} {Mathematics}, 2019.
\newblock \url{https://www.math.uni-bonn.de/people/scholze/Condensed.pdf}.

\bibitem[Sim92]{SimpsonCorrespondence}
Carlos~T. Simpson.
\newblock Higgs bundles and local systems.
\newblock {\em Inst. Hautes \'{E}tudes Sci. Publ. Math.}, (75):5--95, 1992.

\bibitem[SW20]{SW_Berkeley}
Peter Scholze and Jared Weinstein.
\newblock {\em $p$-adic geometry, {{U}C} {B}erkeley course notes}.
\newblock Annals of Mathematics Studies. Princeton University Press, Princeton,
  NJ, 2020.

\bibitem[Wü20]{wuerthen_vb_on_rigid_var}
Matti Würthen.
\newblock {V}ector bundles with numerically flat reduction on rigid analytic
  varieties and $p$-adic local systems.
\newblock {\em Preprint, arXiv:1910.03727}, 2020.

\end{thebibliography}

	\noindent	
	Lucas Mann \hfill Ben Heuer \\
	Mathematisches Institut \hfill Mathematisches Institut \\
	Universit\"at Bonn \hfill   Universit\"at Bonn \hfill 
	\\
	Endenicher Allee 60 \hfill 
	Endenicher Allee 60 \hfill
	\\
	53115 Bonn \hfill 
	53115 Bonn \hfill 
	\\
	mannluca@math.uni-bonn.de \hfill heuer@math.uni-bonn.de
	\\[2ex]
	\noindent
	Annette Werner\\
	Institut f\"ur Mathematik \\
	Goethe-Universit\"at Frankfurt\\
	Robert-Mayer-Str. 6-8\\
	60325 Frankfurt am Main\\
	werner@math.uni-frankfurt.de
\end{document}